%% file: 1-bridge.tex
\documentclass{article}

\usepackage{amsmath}
\usepackage{amsthm}
\usepackage{thmtools}
\usepackage{amsfonts}
\usepackage{graphicx}
\usepackage{subfigure}
\usepackage{hyperref}
\usepackage{xspace}
\usepackage{parskip}
\usepackage{color}
\usepackage{tikz}
\usepackage{verbatim}

\usetikzlibrary{decorations}
\usetikzlibrary{decorations.pathreplacing}

\graphicspath{{figures/}}

\makeatletter
\def\thm@space@setup{%
  \thm@preskip=\parskip \thm@postskip=0pt
}
\makeatother

\declaretheorem[parent=section]{lemma}
\declaretheorem[sibling=lemma]{theorem}
\declaretheorem[sibling=lemma, name=Proposition]{prop}

\declaretheorem[sibling=lemma]{remark}
\declaretheorem[sibling=lemma]{corollary}

\usepackage{xspace}

\newcommand{\boundary}{\partial}
\newcommand{\set}[1]{\left\{#1\right\}}

\newcommand{\closure}[1]{\overline{#1}}
\newcommand{\scc}{simple closed curve\xspace}
\newcommand{\sccs}{simple closed curves\xspace}

\newcommand{\interior}[1]{\operatorname{int} #1}
\newcommand{\co}{\colon}
\newcommand{\R}{\mathbb{R}}
\newcommand{\Q}{\mathbb{Q}}
\newcommand{\algint}[2]{|#1\cdot#2|}
\newcommand{\ie}{\textit{i.e.}\@\xspace}
\newcommand{\cf}{\textit{cf.}\@\xspace}
\newcommand{\ceil}[1]{\left\lceil#1\right\rceil}
\newcommand{\floor}[1]{\left\lfloor#1\right\rfloor}
\newcommand{\SC}{Scharlemann cycle\xspace}
\newcommand{\Z}{\mathbb{Z}}
\newcommand{\presentation}[2]{\left\langle#1\colon#2\right\rangle}

\newcommand{\inv}[1]{\overline{#1}}

\newcommand{\defn}[1]{\emph{#1}}

\begin{document}
\title{Knots in handlebodies with nontrivial handlebody surgeries}
\author{R.\ Sean Bowman}
\date{\today}
\maketitle
\begin{abstract}
  We give examples of knots in a genus 2 handlebody which have
  nontrivial Dehn surgeries yielding handlebodies and show that these
  knots are not $1$--bridge.
\end{abstract}

\section{Introduction}
Let $K$ be a knot in a handlebody $H$ of genus $g$.  It is a natural
question to ask when $K$ has a nontrivial Dehn surgery yielding a
handlebody.  When $K$ is isotopic into $\boundary H$, there are
infinitely many surgeries on $K$ yielding handlebodies homeomorphic to
$H$.  Berge and Gabai~\cite{Berge,Gabai90} have given examples of
handlebodies of genus 2 containing knots which are not isotopic to the
boundary yet have nontrivial handlebody surgeries.  More recently,
Frigerio, Martelli, and Petronio have given infinitely many examples
of knots in handlebodies of every genus $g>1$ which have exactly three
handlebody surgeries~\cite{frigerio}.

We say that $K$ is $1$--bridge in $H$ if $K$ is isotopic to
$\alpha\cup\beta$, where $\alpha\subseteq\boundary H$ is an arc,
$\beta$ is properly embedded in $H$, and there is an arc
$\beta'\subseteq\boundary H$ so that $\beta\cup\beta'$ bounds a disk.
Berge~\cite{Berge91} gave examples of $1$--bridge knots in solid tori
which are not isotopic to the boundary and have surgeries yielding
solid tori.  Together with results of Gabai~\cite{Gabai89}, who showed
that knots in solid tori with solid tori surgeries must either be
isotopic to the boundary or $1$--bridge, this leads to a classification
of such knots when $g=1$~\cite{Berge87,Gabai90}.

The examples of knots in genus $g>1$ handlebodies with handlebody
surgeries given in~\cite{Berge,frigerio,Gabai90} are all $1$--bridge
knots.  Wu~\cite{Wu93} conjectured that Gabai's result should hold in
this case, that is, if $K$ has a nontrivial surgery yielding a
handlebody homeomorphic to $H$, then $K$ is $1$--bridge in $H$.  In
this paper we give examples of knots in a genus 2 handlebody which
disprove this conjecture.  The exteriors of these knots are
atoroidal, anannular, and have exactly two fillings homeomorphic to
handlebodies.  

Specifically, we construct a family of knots $K^L_{p,q}$ in a genus
$2$ handlebody $M^L_{p,q}$ parametrized by two integers $p,q$ and a
$2$--bridge knot $L$.  The handlebody $M^L_{p,q}$ contains an embedded
separating $3$--punctured sphere $P$ which cuts $M$ into two
handlebodies $J$ and $H$, with $K\subseteq H$.  We show

\begin{theorem}
  The exteriors of the knots $K_{p,q}^L\subseteq M_{p,q}^L$ are
  atoroidal, anannular, and have exactly two fillings homeomorphic to
  handlebodies.  Furthermore, $K_{p,q}^L$ is not $1$--bridge in
  $M_{p,q}^L$ when either
  \begin{enumerate}
  \item $q=p+1$ and $p>1$, or
  \item $q=2p\pm 1$, $p>1$, and $q>3$.
  \end{enumerate}
\end{theorem}
\begin{proof}
  The first properties follow from~\autoref*{K-atoroidal}
  and~\autoref*{two-handlebody-fillings}.  For the first case, let
  $(M, H, K) = (M_{p,p+1}^L, H_{p,p+1}^L, K_{p,p+1}^L)$ and suppose
  that $K$ is $1$--bridge in $M$.  By~\autoref*{1-bridge-iff-1-tunnel}
  there is a tunnel $t$ so that $M'=\closure{M\setminus (K\cup t)}$ is
  a handlebody.  \autoref*{t-cap-P-is-empty} shows that we may take
  $t$ to lie entirely in
  $H$.  
  According to~\autoref*{sfce-in-hbody} the surface $P$ is
  $\boundary$--compressible in $M'$, and it must $\boundary$--compress
  in the handlebody $H'=\closure{H\setminus N(K\cup t)}$
  by~\autoref*{P-in-J}.  But by~\autoref*{easy-case}, such
  $\boundary$--compressing disks do not exist.

  For the second case, the proof proceeds as above except
  that~\autoref*{hard-case} shows that $P$ is
  $\boundary$--incompressible.  This
  contradicts~\autoref*{sfce-in-hbody}.
\end{proof}

The paper is organized as follows: in section 2 we review some basic
definitions.  Section 3 describes a characterization of $1$--bridge
knots as those admitting a certain type of tunnel.  In section 4, we
describe a large family of knots in handlebodies and examine some
properties of these knots.  The results of section 5 show that these
knots admit exactly two Dehn surgeries yielding handlebodies.  In
section 6, we prove that infinitely many of the knots are distinct up
to homeomorphism of the handlebody.  In section 7 we examine the
position of tunnels described in section 3, and in section 8 we prove
some technical lemmas about $\boundary$--compressing disks in an
associated manifold.  Finally, in section 9 we show that two
subfamilies of the knots described in section 4 are not $1$--bridge.

The author would like to thank the referee, Robert Myers, and
especially John Luecke for many valuable conversations and
suggestions.  This work is partially supported by NSF RTG grant
DMS-0636643.

\section{Definitions}
Let $S$ be a surface with boundary.  An arc $\alpha$ properly embedded
in $S$ is called \defn{essential} if it does not cobound a disk with a
subarc of $\boundary S$.  Otherwise, it is \defn{trivial} or
\defn{$\boundary$--parallel}.  A disk $D$ properly embedded in a
$3$--manifold $M$ is called \defn{essential} if $\boundary D$ does not
bound a disk in $\boundary M$.  We call an essential disk in a
handlebody a \defn{meridian disk}.

If $S$ is a surface properly embedded in a $3$--manifold $M$, we say
that $S$ is compressible if there is an essential \scc in $S$ which
bounds an embedded disk in $M$.  Otherwise $S$ is
\defn{incompressible}.  An incompressible surface $S$ is
\defn{$\boundary$--compressible} if either $S$ is a disk which is
isotopic into $\boundary M$ or there exists a disk $D$ in $M$ such
that $D\cap S=\alpha$ is an essential arc in $S$, $D\cap\boundary
M=\beta$ is an arc in $\boundary D$, $\boundary D=\alpha\cup\beta$,
and $\alpha\cap\beta=\boundary\alpha=\boundary\beta$.  Otherwise, $S$
is $\boundary$--incompressible in $M$.

If $S$ is an embedded subsurface of $\boundary M$, we say that $S$ is
$\boundary$--compressible if there is an essential disk $D$ in $M$
such that $\boundary D$ intersects $S$ in a single arc essential in
$S$.  Otherwise $S$ is $\boundary$--incompressible.

Let $M$ be an orientable $3$--manifold with boundary, and let $D_0$,
and $D_1$ be a pair of disjoint disks in $\boundary M$.  We say that
the manifold $M\cup_{D_0\cup D_1} D^2\times I$ is obtained from $M$ by
\defn{attaching a $1$--handle}, where $D^2\times \boundary I$ is
identified with $D_0\cup D_1$ in such a way as to obtain an orientable
manifold.  The \defn{cocore} of the $1$--handle $D^2\times I$ is
$D^2\times\set{1/2}$.

Let $\alpha$ be a \scc in the boundary of a $3$--manifold $M$, and let
$A$ be a regular neighborhood of $\alpha$ in $\boundary M$.  We say
that the manifold $M\cup_A (D^2\times I)$ is obtained from $M$ by
\defn{attaching a $2$--handle} along $\alpha$, where $A$ and
$\boundary D^2\times I$ are identified.  We let $M[\alpha]$ denote the
manifold obtained by attaching a $2$--handle to $M$ along $\alpha$,
and call the curve $\alpha$ the \defn{attaching curve} of the
2--handle.

We say that a \scc $\gamma$ in the boundary of handlebody $H$ is
\defn{primitive} if attaching a 2--handle to $H$ along $\gamma$ yields
another handlebody.  Note that this is equivalent to the existence of a
meridian disk for $H$ which meets $\gamma$ in exactly one point.

If $F$ is a submanifold of $M$, denote by $N_M(F)$ a closed regular
neighborhood of $F$ in $M$.  When it is clear from the context what
$M$ is, we simply write $N(F)$.  

Let $\alpha$ be a \scc in a surface $S$ embedded in a $3$--manifold
$M$.  The isotopy class in $\boundary N(\alpha)$ of the curves
$\boundary N(\alpha)\cap S$ is called the \defn{surface slope} of $\alpha$
with respect to $F$.  

If $K$ is a knot embedded in $S^3$, denote $E(K)=\closure{S^3\setminus
  N(K)}$, the \defn{exterior of $K$}.  More generally, if $K$ is a
knot embedded in some $3$--manifold $M$, the \defn{exterior of $K$ in
  $M$} is the manifold $\closure{M\setminus N(K)}$.  

We adopt the language of~\cite{Wu12} for rational tangles.  A
\defn{tangle} is a pair $(B, t, m)$ where $B$ is a $3$--ball,
$t=t_1\cup t_2$ is a pair of arcs properly embedded in $B$, and $m$ is
a \scc on $\boundary B$ which divides $\boundary B$ into two disks
each containing two points of $\boundary t$.  We will say that two
tangles $(B, t, m)$ and $(B, t', m')$ are equivalent if there is a
homeomorphism of the triples which is the identity on $\boundary B$.

A tangle $(B, t, m)$ is \defn{rational} if $(B, t, m)$ is equivalent
to the trivial tangle $(D^2, \set{x, y})\times I$ for $x, y\in D^2$.
Here $m$ is the vertical circle which bounds a disk separating $t_1$
and $t_2$ in $B$.  In this case, $t$ is isotopic rel $\boundary t$ to
a pair of arcs on $\boundary B$ of slope $r/s\in \Q\cup\set{1/0}$.

For a rational number $q$, denote the ceiling of $q$ by $\ceil{q}$ and
the floor of $q$ by $\floor{q}$.

\section{Characterization of 1--bridge knots}\label{chapter:knots}
In this section we give a characterization of $1$--bridge knots as
those possesing a certain type of tunnel.  

\begin{lemma}\label{defn-of-unknotted-arc}
  Let $\beta\subseteq H$ be an arc properly embedded in a handlebody.
  Then there is an embedded disk $D$ such that $\boundary D =
  \beta\cup\beta'$, $\beta'\subseteq\boundary H$, and $\boundary \beta
  = \boundary \beta'$, iff $\closure{H\setminus N(\beta)}$ is a
  handlebody.
\end{lemma}
\begin{proof}
  Suppose that there is a disk $D$ as in the claim.  After cutting
  $H'=\closure{H\setminus N(\beta)}$ along $D\cap H'$ we obtain an
  handlebody homeomorphic to $H$.  On the other hand, the space $H$ is
  obtained from $H'$ by adding a 2--handle.  Let $\gamma$ be the
  attaching curve of the 2--handle, and note that $\gamma$ is
  primitive in $H'$.  Therefore there is a disk $D'\subseteq H'$
  meeting $\alpha$ exactly once.  This disk extends to a disk
  $D\subseteq H$ with $\boundary D=\alpha\cup\beta$ and
  $\alpha\subseteq \boundary H$.
\end{proof}

If an arc $\beta$ satisfies these conditions we say that $\beta$ is
\defn{unknotted}.

Let $H$ be a handlebody of genus $g>0$.

\begin{prop}\label{1-bridge-iff-1-tunnel}
  A knot $K\subseteq H$ is 1--bridge iff there is an embedded arc
  $t\subseteq H$ so that $\boundary t =\set{a, b}$, $K\cap t =
  \set{a}$, $\boundary H\cap t = \set{b}$, and $\closure{H\setminus
    N(t\cup K)}$ is a handlebody.
\end{prop}
\begin{proof}
  Suppose $K$ is isotopic to $\alpha\cup\beta$, where
  $\alpha\subseteq\boundary H$ and $\beta$ is unknotted.  Using a
  collar neighborhood of $\boundary H$, push $\alpha$ and $\beta$
  inside $H$ slightly to obtain an isotopic curve $\alpha'\cup\beta'$.
  A point $p\in\interior{\alpha}$ traces an arc $t$ under this
  isotopy.  We can use $t$ to push the arc $\alpha'$ to the boundary,
  showing that $\closure{H\setminus N(\alpha'\cup\beta'\cup t)}\cong
  \closure{H\setminus N(\beta)}$.  However, $\beta$ is unknotted, so
  this last space is a handlebody.

  On the other hand, if $t$ is an arc as in the proposition, we may
  slide $K$ along $t$ in a neighborhood of $t$ to the form
  $\alpha\cup\beta$, where $\alpha\subseteq \boundary H$ and $\beta$
  is a properly embedded arc in $H$.  Then $\closure{H\setminus
    N(\beta)}\cong \closure{H\setminus N(K\cup t)}$, and we can think
  of $N(\beta)$ as a $2$--handle attached to $\closure{H\setminus
      N(\beta)}$ along a curve $\gamma$ to obtain $H$.  Clearly
  $\gamma$ is primitive in $\closure{H\setminus N(\beta)}$, so there
  is a meridian disk $D$ of $\closure{H\setminus N(\beta)}$ whose
  boundary meets $\gamma$ exactly once.  This shows that $\beta$ is
  unknotted in $H$, and therefore $K$ is 1--bridge.  
\end{proof}

\section{The Knots}\label{section:knots}
In this section we present a large family of knots in a genus 2
handlebody which have nontrivial handlebody surgeries.  The knots are
constructed by gluing a genus two handlebody containing a knot to
another genus 2 handlebody along a $3$--punctured sphere.  This
$3$--punctured sphere becomes an embedded incompressible surface which
is $\boundary$--incompressible in the complement of the knot.
Furthermore, the knot exteriors contain no essential tori or annuli.
We then prove some facts about these knots which will be useful later.

Let $m$ and $l$ be a meridian and longitude, respectively, in the
boundary of a solid torus $D^2\times S^1$, and let $c$ be the core
curve of the solid torus.  Identify $\closure{D^2\times S^1\setminus
  N(c)}$ with $T^2\times I$ where $I$ is the interval $[0,1]$ and
$T^2\times\set{1}=\boundary N(c)$.  Let $K$ be a $(p,q)$ curve in
$T^2\times\set{0}$, \ie, a curve meeting $m$ algebraically and
geometrically in $p$ points and meeting $l$ algebraically and
geometrically in $q$ points.  Let $e$ be a curve in $T^2\times \set{0}$
parallel to $K$ with $K\cap e=\emptyset$, and let $\pi$ be a curve in
$T^2\times\set{0}$ meeting $K$ exactly once and $e$ exactly once.  The
annuli $e\times I$ and $\pi\times I$ meet in a single arc $a$ and
restrict to disks in the handlebody $H=\closure{T^2\times I\setminus
  N(a)}$.  The disk $D=(\pi\times I)\cap H$ marks $K$ as a primitive
curve in $\boundary H$, and $E=(e\times I)\cap H$ gives a
nonseparating disk which is disjoint from $K$.  

Push $K$ to the interior of $H$ and denote by $(H_{p,q}, K_{p,q})$ the
pair of handlebody and knot obtained by this procedure.  We will
consider two pairs $(H, K)$ and $(H', K')$ to be equivalent if there
is an orientation preserving homeomorphism $f\co H\to H'$ such that
$f(K) = K'$.  If this is the case we write $(H,K)\cong (H', K')$.

Let $x$ be the boundary component of $m\times I$ lying on $\boundary
N(c)\cap H$.  Then $x$ meets $\boundary E$ in $p$ points.  Let $s$ be
a subarc of $\boundary E$ with one endpoint on $x$ and the other on
$l$.  Let $P$ be a neighborhood of $x\cup s\cup l$ in $\boundary H$.
This surface is a $3$--punctured sphere in $\boundary H$ whose
complement $R=\closure{\boundary H\setminus P}$ is also a
$3$--punctured sphere.  Call $\boundary_1 P$ the component of
$\boundary P$ which is homotopic to $x$ in $\boundary H$, $\boundary_2
P$ the component which is homotopic to $l$, and $\boundary_3P$ the
third boundary component.  Note that $H$ is homeomorphic to $P\times
I$, and in particular, $P$ is incompressible in $H$.

Choose curves $m'$ and $l'$ in $T^2\times \set{0}$ parallel to $m$ and
$l$, respectively, so that $D_m=(m'\times I)\cap H$ and $D_l=(l'\times
I)\cap H$ are meridian disks in $H$ meeting $K$ in $p$ and $q$ points.
Note that the pair is a complete disk system for $H$, \ie, cutting $H$
along these two disks yields a ball.

Pictured in~\autoref*{fig:involution} is $H_{3,4}$ together with
$\boundary E$ and $P$ (shaded, green).  Also shown is the axis of an
involution $\tau$ which interchanges $P$ and $R$ and leaves $K$
invariant. Note that there were two choices for the arc $s$ above.
The involution $\tau$ shows that the pair $(H,P)$ is well defined.


\begin{figure}[h!tb]
  \begin{center}
    \def\svgwidth{0.9\textwidth}
    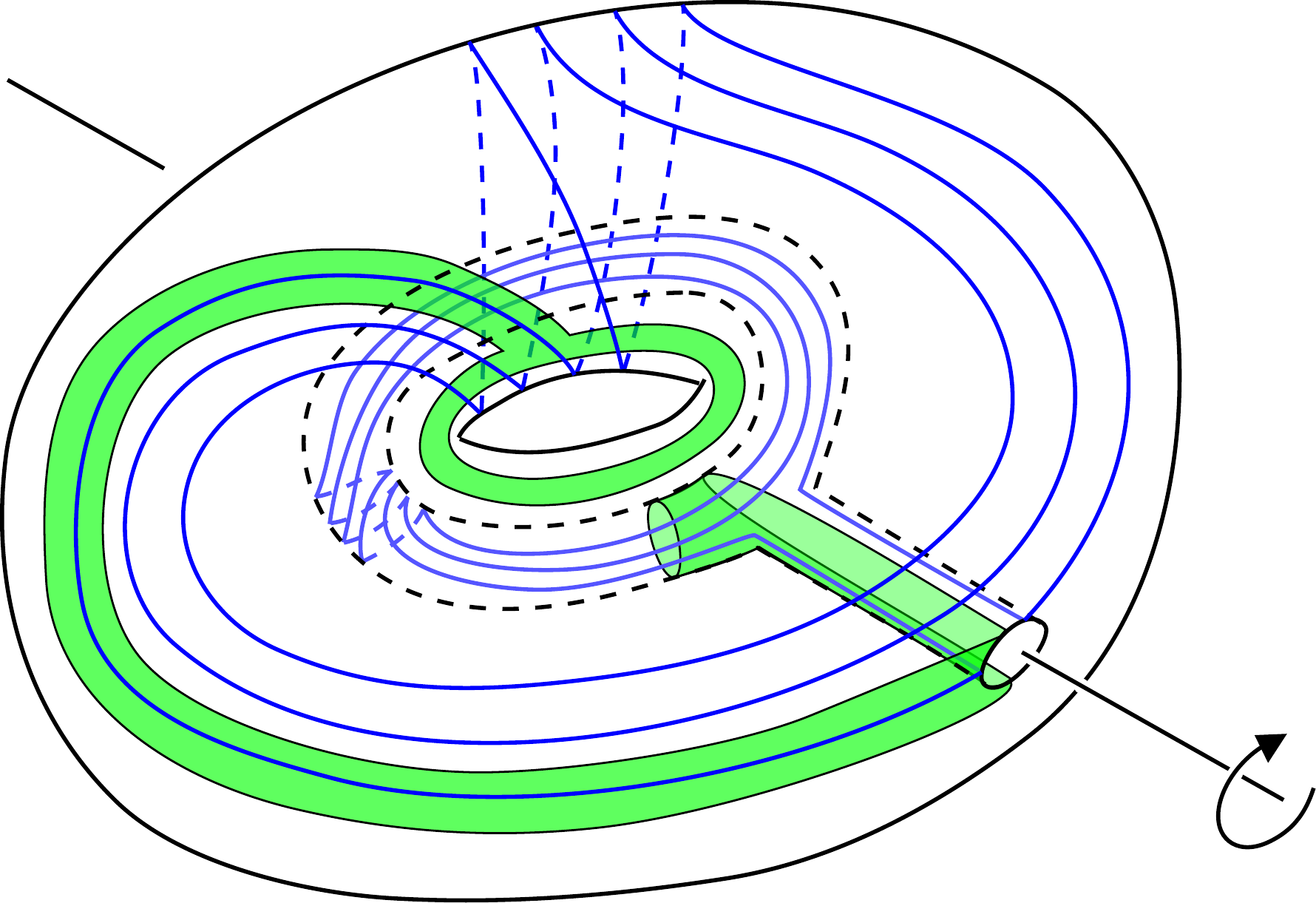
  \end{center}
  \caption{$H_{3,4}$ together with $\boundary E$ and $P$}
  \label{fig:involution}
\end{figure}

There is another involution $\sigma$ on $H$ which interchanges $D_m$
and $D_l$.  Under this involution, shown
in~\autoref*{fig:involution2}, $P$ is invariant: $\boundary_3 P$ is
sent to itself, and $\boundary_1 P$ and $\boundary_2 P$ are swapped.
Note that the attaching curve of the $2$--handle $N(a)$ can be taken
to be invariant under this involution.  The involution shows that
$(H_{p,q},K_{p,q})\cong (H_{q,p},K_{q,p})$.

\begin{figure}[h!tb]
  \begin{center}
    \def\svgwidth{0.7\textwidth}
    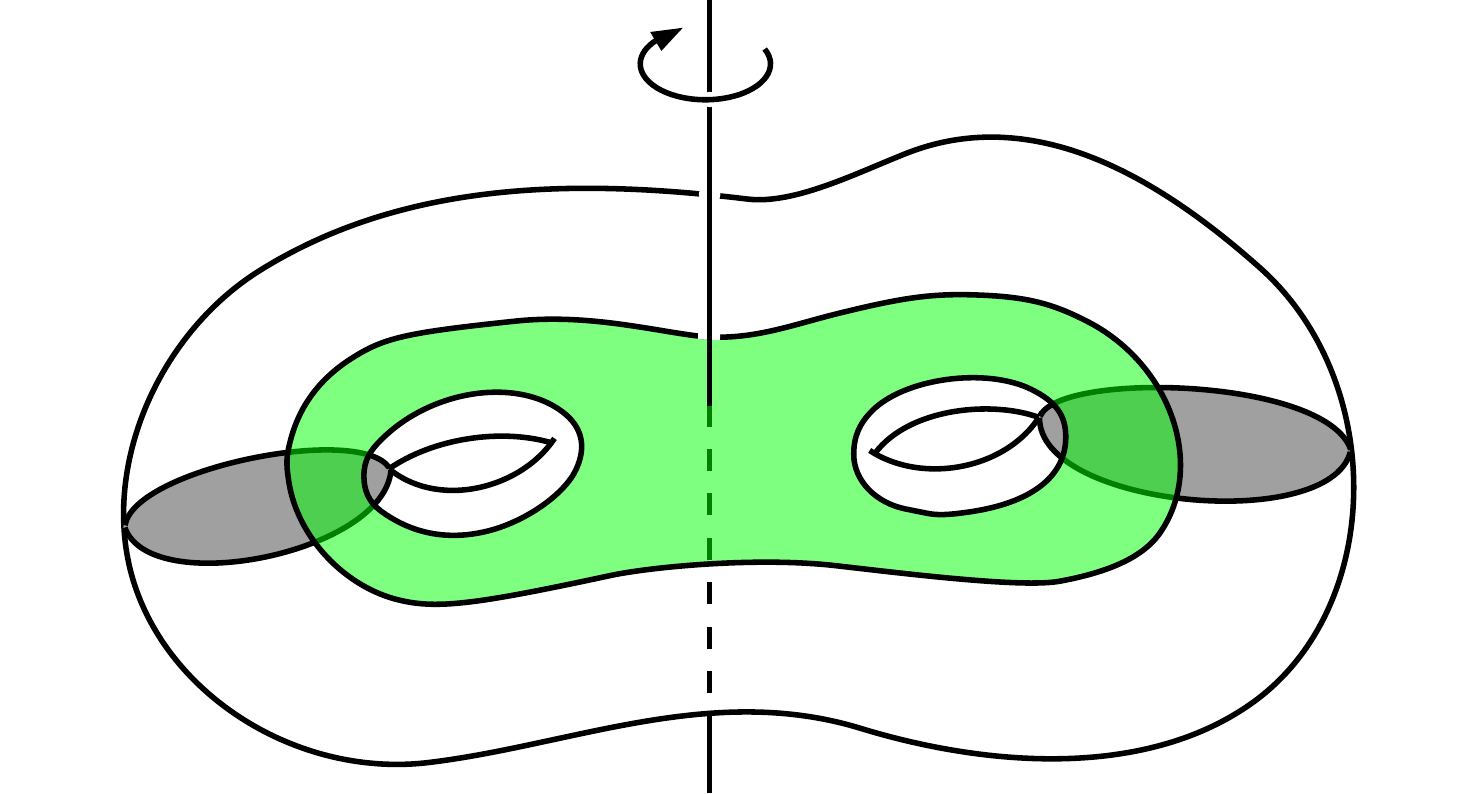
  \end{center}
  \caption{Another view of $H_{p,q}$ with an involution}
  \label{fig:involution2}
\end{figure}



We obtain the knots we wish to study by gluing $H_{p,q}$ to another
handlebody $J$ along $P$.  Here we describe the construction of $J$.
Let $L$ be a nontrivial 2--bridge knot in $S^3$, in 2--bridge position
with respect to some height function $h\co S^3\to \R$.  Then $L$ has
an unknotting tunnel $t$ connecting its two maxima~\cite{Kobayashi99},
that is, a tunnel so that $J=\closure{S^3\setminus N(L\cup t)}$ is a
handlebody.  The knot exterior $E(L)$ is obtained from $J$ by
attaching the 2--handle $N(t)$ along an attaching curve
$\gamma\subseteq\boundary J$.  Let $P'$ be a 3--punctured sphere
embedded in $\boundary J$ with $\boundary P' =
\alpha\cup\beta\cup\gamma$, where $\alpha$ and $\beta$ are meridians
of $L$.

If we attach a $2$--handle to $J$ along $\alpha$ or $\beta$, we obtain
a new knot exterior whose knot has one maximum and one minimum with
respect to $h$.  This must be the unknot, and therefore $\alpha$ and
$\beta$ are primitive curves.  On the other hand, $\gamma$ is not
primitive, since attaching a $2$--handle along it gives the exterior
of a nontrivial knot.  Note that $\alpha$ and $\beta$ are
\defn{jointly primitive} in the sense that $J[\alpha, \beta]$ is a
$3$--ball.  Therefore we may write them as the cores of the left and
right handles of $J$.  After embedding $J$ in $S^3$ so that both
$\alpha$ and $\beta$ bound disks in $S^3\setminus H$, we see that
there is an arc $c$ properly embedded in $P$ which connects $\alpha$
and $\beta$.  This arc is unique up to isotopy.  The sphere
$S=\boundary J[\alpha, \beta]$ is a genus zero Heegaard surface which
divides $S^3$ into two balls, $B_0$ and $B_1$.  Assume $J\subseteq
B_0$ and call the cocores of the attached $2$--handles $a$ and $b$.
These are unknotted arcs in $B_0$.

The arc $c\subseteq\boundary J$ extends to an arc in $S$ meeting one
endpoint of $a$ and one endpoint of $b$.  By abuse of notation we will
also call this arc $c$.  Note that there is an arc $c'\subseteq
S\setminus(\boundary a\cup\boundary b)$ such that $c\cap c'=\emptyset$
and $a\cup b\cup c\cup c'$ is a knot $K_0$ in $S^3$.  Pushing $c$ and
$c'$ slightly inside $B_1$, we see that $S$ is a bridge surface for
$K_0$.  Furthermore, the exterior of $K_0$ in $S^3$ is homeomorphic to
the space obtained by attaching a $2$--handle to $J$ along $\gamma$.
We may therefore regard the pair $(J, P)$ as the exterior of a $m/n$
rational tangle with $P$ as the twice punctured disk on the left half
sphere.  By viewing the pair $(J, P)$ in this way, we will be able to
apply the results of~\cite{Wu12} regarding disks in $J$.

Let $J$ and $P'$ be a handlebody and pants pair as constructed above,
and let $(H_{p,q}, K_{p,q})$ be the handlebody and knot pair obtained
from a $(p,q)$ curve in the boundary of a solid torus as above.
Identify $P$ and $P'$ so that $\boundary_1P$ and $\boundary_2P$ are
identified with $\alpha$ and $\beta$. Note that since $H_{p,q}\cong
P\times [0,1]$, the resulting space $M^L_{p,q}=J\cup H_{p,q}$ is a
handlebody and does not depend on the choice of identification of $P$
and $P'$ up to equivalence.  

When $J$ is viewed as the exterior of an $m/n$ rational tangle as
above, it is easy to see that there is an involution of $J$ which
swaps $\alpha$ and $\beta$ and leaves $\gamma$ invariant.  Together
with the involutions $\tau$ and $\sigma$ considered earlier, we see
that the pair $(M^L_{p,q},K_{p,q})$ is well defined and that
$(M^L_{p,q},K_{p,q})\cong(M^L_{q,p},K_{q,p})$.  Therefore from now on
we will always assume that $p<q$.  We consider $K$ as a knot in the
interior of $M^L_{p,q}$ and write $(M^L_{p,q},K^L_{p,q})$ when we need
to emphasize the dependence on $L$, $p$, or $q$.

In \autoref*{section:distinct} we will show that there are infinitely
many distinct such knots.  Later we will show that if $L$ is a
nontrivial $2$--bridge knot and $n>1$ is an integer,
$K^L_{n,n+1}\subseteq M^L_{n,n+1}$ and $K^L_{n,2n\pm 1}\subseteq
M^L_{n,2n\pm 1}$ are not $1$--bridge knots.  In order to do this, we
need a few facts about $H$ and $J$.

Fix integers $q>p>1$ and a nontrivial $2$--bridge knot $L$, and let
$(H,K) = (H^L_{p,q}, K^L_{p,q})$. 

\begin{lemma}\label{P-in-J}
  The surface $P'$ is incompressible and $\boundary$--incompressible in $J$.
\end{lemma}
\begin{proof}
  Suppose that $(J, P')$ arises from the exterior of an $m/n$ rational
  tangle as described above.  Let $D$ be a compressing disk or a
  $\boundary$--compressing disk whose boundary intersects $\boundary
  P'$ minimally.  Following~\cite{Wu12}, we say that $D$ is an $(r,
  s)$ disk if it meets $\boundary_1 P'\cup \boundary_2 P'$ in $r$
  points and $\boundary_3 P'$ in $s$ points.

  If $D$ is a compressing disk for $P'$, then it is a $(0, 0)$ disk
  and~\cite[Lemma 2.3]{Wu12} implies that $n = 0$.  If $D$ is a
  $\boundary$--compressing disk for $P'$, then $D$ is either a $(1,
  1)$ disk, a $(2, 0)$ disk, or a $(0,2)$ disk.  The same lemma
  implies that either $n=1$ or $\inv{m}$, the mod $n$ inverse of $-m$,
  is nonzero.  Since these are all contradictions to the fact that $L$
  is nontrivial, $P'$ must be incompressible and
  $\boundary$--incompressible in $J$.
\end{proof}

\begin{lemma}\label{attaching-2-h}
  After attaching a $2$--handle along $\boundary_iP$,
  $K\subseteq H$ becomes a knot in a solid torus which
  meets a meridian disk algebraically
  \begin{itemize}
    \item $p$ times if $i=1$,
    \item $q$ times if $i=2$, and
    \item $q-p$ times if $i=3$.
  \end{itemize}
\end{lemma}
\begin{proof}
  The first two are immediate from the construction.  To see the
  third, consider $H$ embedded in $S^3$ as shown
  in~\autoref*{fig:involution}.  After attaching a $2$--handle along
  $\boundary_3P$, we obtain a solid torus in $S^3$ whose complement is
  also a solid torus.  The linking number of a core of the
  complementary solid torus with $K$ gives the appropriate
  algebraic intersection number, which is easily seen to be $q-p$.
\end{proof}

\begin{lemma}\label{E-is-unique}
  Up to isotopy, the disk $E$ is the unique nonseparating essential
  disk in $\closure{H\setminus N(K)}$.
\end{lemma}
\begin{proof}
  Suppose there are two nonisotopic nonseparating essential disks in
  $H$, $E_1$ and $E_2$, which do not meet $K$.  Among all such pairs,
  choose the one which minimizes $|E_1\cap E_2|$.  Let $D$ be an
  essential disk in $H$ meeting $K$ exactly once and with $|D\cap
  E_1|$ minimal.  An innermost curve/outermost arc argument shows that
  $D\cap E_1=\emptyset$.




  After compressing $H$ along $E_1$ we obtain a solid torus containing
  a knot $K$.  The disk $D$ shows that the boundary of the resulting
  solid torus is incompressible in the complement of $K$  and therefore
  $E_1\cap E_2\neq\emptyset$.

  We may assume that $E_1\cap E_2$ consists of arcs; let $\alpha$ be
  one outermost in $E_1$ which cuts off a subdisk $\delta$ of $E_1$ so
  that the interior of $\delta$ does not meet $E_2$.  Surgering $E_2$
  along $\delta$ we obtain two essential disks, one of which must be
  nonseparating.  Call this disk $E_2'$.  The disk $E_2'$ is disjoint
  from $E_2$ and has at least one fewer arc of intersection with
  $E_1$.  Therefore both pairs $(E_2, E_2')$ and $(E_1, E_2')$ are
  parallel, and so $E_1$ is parallel to $E_2$.
\end{proof}

\begin{lemma}\label{minimum-of-E-cap-P}
  The disk $E$ meets $P$ in $p+q-1$ arcs and cannot be isotoped to
  reduce this intersection.  
\end{lemma}
\begin{proof}
  From the description it is clear that we can choose $P$ so that
  $\boundary E\cap P$ consists of a single arc joining $\boundary_1P$
  and $\boundary_2P$, $p-1$ parallel arcs joining $\boundary_1 P$ and
  $\boundary_3 P$, and $q-1$ parallel arcs joining $\boundary_2 P$ and
  $\boundary_3 P$.  Similarly, $\boundary E\cap R$ consists of a
  single arc joining $\boundary_1P$ and $\boundary_2P$, $p-1$ parallel
  arcs joining $\boundary_1 P$ and $\boundary_3 P$, and $q-1$ parallel
  arcs joining $\boundary_2 P$ and $\boundary_3 P$.

  Note that $|\boundary E\cap P| = \frac{1}{2}\left(|\boundary E\cap
    \boundary_1 P| + |\boundary E\cap \boundary_2 P| + |\boundary
    E\cap \boundary_3 P|\right)$.  If we could reduce $|\boundary
  E\cap P|$, we would be able to reduce one of $|\boundary
  E\cap\boundary_i P|$.  In this case, there would be a bigon of
  intersection of $\boundary E$ and $\boundary_i P$ on $\boundary H$,
  and an innermost such bigon of intersection would lie in $P$ or $R$,
  giving a trivial arc.  However, each of the arcs above is
  nonseparating, and so this is the minimal intersection of $P$ and
  $\boundary E$.
\end{proof}

The disk $E$ is pictured in~\autoref*{fig:E} together with labels for
subarcs of $\boundary E$.  These labels correspond to arcs
in~\autoref*{fig:PR}, which shows the surfaces $P$ and $R$ along with
the arcs $P\cap\boundary E$ and $R\cap\boundary E$ given
by~\autoref*{minimum-of-E-cap-P}.  The bold arcs represent a
collection of parallel arcs.  There are three groups of arcs in $P$:
there is a single arc $a^P$ with endpoints on $\boundary_1P$ and
$\boundary_2P$.  There are $p-1$ arcs
$d^P_{\inv{q}},d^P_{2\inv{q}},\dots,d^P_{(p-1)\inv{q}}$ with endpoints
on $\boundary_1P$ and $\boundary_3P$.  Here $\inv{q}$ is the inverse
of $q$ in $\Z_p$, and the indices are taken modulo $p$, so that
$d^P_{-1}=d^P_{p-1}$.
There are $q-1$ arcs $b^P_{(q-1)\inv{p}},b^P_{(q-2)\inv{p}},\dots,
b^P_{\inv{p}}$ with endpoints on $\boundary_2P$ and $\boundary_3P$.
Here $\inv{p}$ is the inverse of $p$ in $\Z_q$, and the indices are
taken modulo $q$ so that $b^P_{-1}=b^P_{q-1}$.

\begin{figure}[h!tb]
  \begin{center}
    \def\svgwidth{0.5\textwidth}
    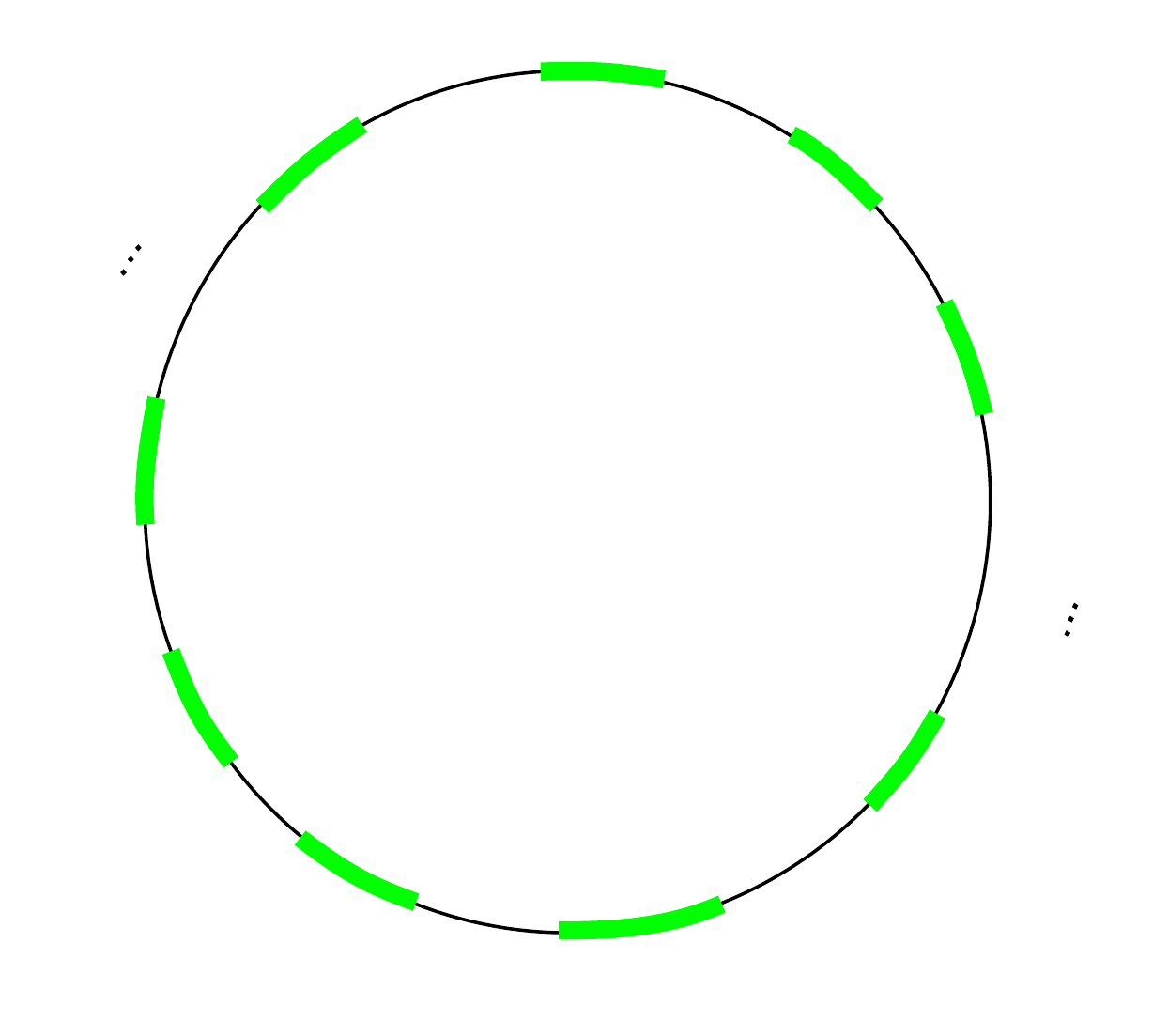
  \end{center}
  \caption{The disk $E$}
  \label{fig:E}
\end{figure}

In $R$, there is a similar collection of arcs.  The arc in $R$ with
endpoints on $\boundary_1P$ and $\boundary_2P$ is called $c^R$, but
otherwise the names are similar.  We will call components of
$P\setminus\boundary E$ and $R\setminus\boundary E$ \defn{regions} of
$P$ and $R$, respectively.  All regions of $P$ and $R$ are rectangles
except for two hexagons, $H_1^P$ and $H_2^P$, in $P$ and two, $H_1^R$
and $H_2^R$, in $R$.  The boundary of $H_1^P$ contains the arcs
$d_{\inv{q}}^P$ and $b_{\inv{p}}^P$, and the boundary of $H_1^R$
contains $d_{\inv{q}}^R$ and $b_{\inv{p}}^R$.  Similarly, the boundary
of $H_2^P$ contains the arcs $d_{(p-1)\inv{q}}^P$ and
$b_{(q-1)\inv{p}}^P$, and the boundary of $H_2^R$ contains
$d_{(p-1)\inv{q}}^R$ and $b_{(q-1)\inv{p}}^R$.  When there is no
danger of confusion we will drop the superscripts.

\begin{figure}[h!tb]
  \begin{center}
    \def\svgwidth{0.8\textwidth}
    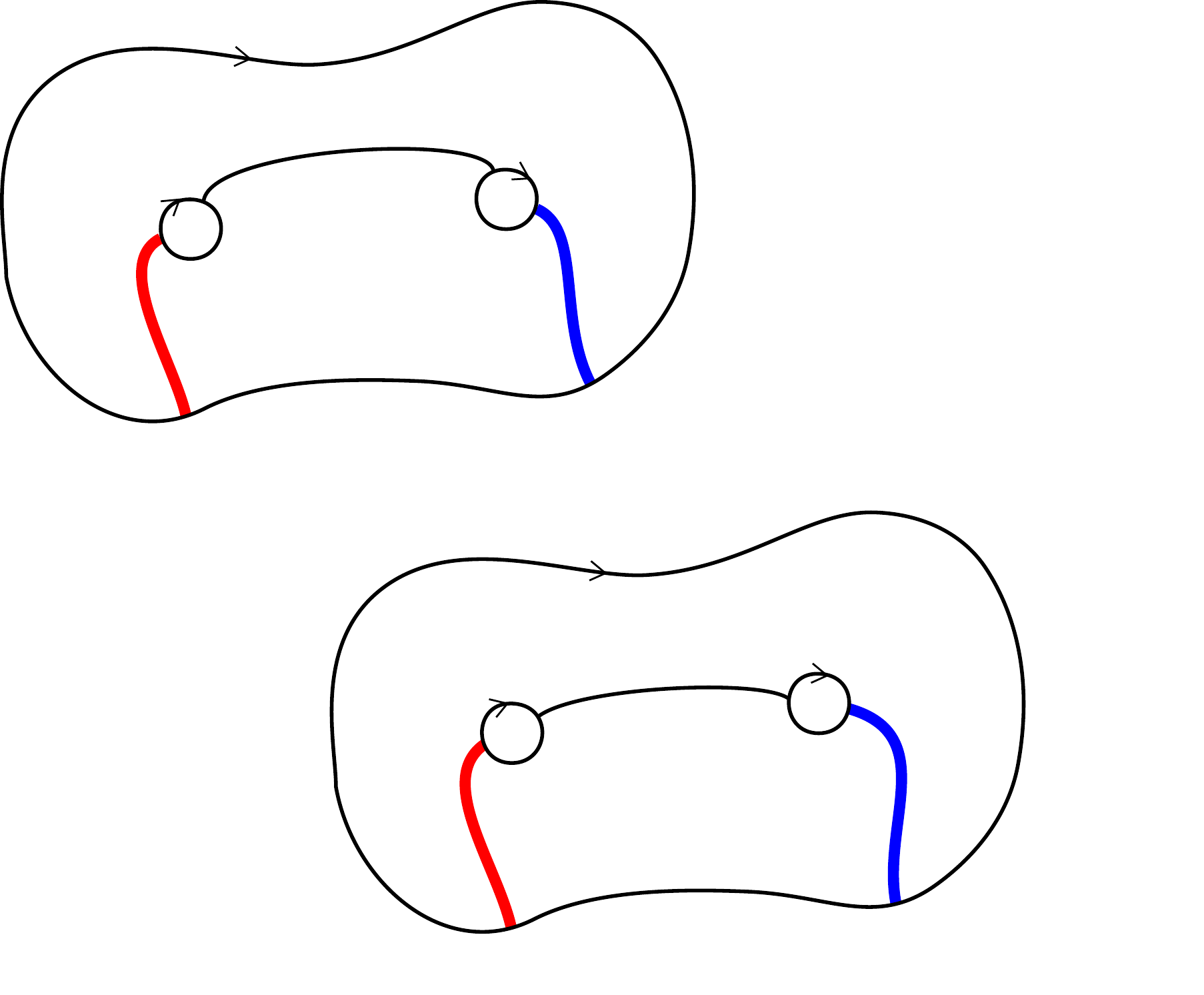
  \end{center}
  \caption{The surfaces $P$ and $R$, and the boundary of $E$}
  \label{fig:PR}
\end{figure}

\begin{remark}\label{P-R-remarks}
There are several things to note about these two figures:
\begin{itemize}
  \item $a^P\cap\boundary_2P = b_1^R\cap\boundary_2P$ and
    $a^P\cap\boundary_1 P = d_{-1}^R\cap\boundary_1P$,
  \item $c^R\cap\boundary_1P = d_1^P\cap\boundary_1P$ and
    $c^R\cap\boundary_2P = b_{-1}^P\cap\boundary_2P$,
  \item $b_i^R\cap\boundary_3P = b_i^P\cap\boundary_3P$,
  \item $d_i^R\cap\boundary_3P = d_i^P\cap\boundary_3P$,
  \item $b_i^P\cap\boundary_2P = b_{i+1}^R\cap\boundary_2P$ for
    $i=1,\dots,q-2$, 
  \item $d_i^R\cap\boundary_1P = d_{i+1}^P\cap\boundary_1P$ for
    $i=1,\dots,p-2$, and
  \item $H_1^P\cap\boundary_3P = H_1^R\cap\boundary_3P$ and
    $H_2^P\cap\boundary_3P=H_2^R\cap\boundary_3P$.
\end{itemize}
\end{remark}

Finally, here are some facts about $H$, $J$, $K$, and $M$ which will
be used later.

\begin{lemma}\label{no-annuli-on-P}
  Let $N$ be $H$ or $J$. Any incompressible annulus $A$ properly
  embedded in $N$ with $\boundary A\subseteq P$ may be isotoped to lie
  entirely in $P$.
\end{lemma}
\begin{proof}
  Let $A$ be such an annulus.  No two components of $\boundary P$ are
  homologous in $H$ since it is a product.  Furthermore, after
  attaching a $2$--handle to $J$ along a component of $\boundary P$,
  the other two components of $\boundary P$ become meridians in a knot
  exterior.  Therefore no two components of $\boundary P$ are
  homologous in $N$, and so $\boundary A$ bounds an annulus
  $B\subseteq P$.  The surface $A\cup B$ is a torus $T$.  Isotop $T$
  slightly inside $N$ to obtain an embedded torus.  There is an
  annulus $C$ so that $\boundary_1C$ lies on $P$ and $\boundary_2C$ is
  a core of $B$.  Since a core of $C$ is isotopic to a component of
  $\boundary P$ in $N$ and $P$ is incompressible in $N$, $C$ must be
  incompressible in $N$.

  A torus in a handlebody must compress to one side.  If $T$ were
  compressible to the side containing $C$, then using an innermost
  disk/outermost arc argument we could find a compressing disk $D$ for
  $T$ not meeting $A$.  Therefore we could isotop $D$ so that
  $\boundary D\subseteq A$.  But $A$ is incompressible, and so $T$
  must be compressible to the side not containing $C$.  We conclude
  that $T$ bounds a solid torus $S$ on this side.

  If the core of $B$ were not longitudinal in $S$, we would
  obtain a reducible manifold after attaching a $2$--handle to $N$
  along this curve.  Examining $J$ and $H$ it is clear that this does
  not happen, and therefore we can isotop $A$ to $B$ through $S$.
\end{proof}

\begin{lemma}
  Every $\boundary$--compressing disk for $P$ in $H\cong P\times I$ is
  isotopic to one of the six disks of the form $\lambda\times I$ for
  an essential arc $\lambda\subseteq P$.
\end{lemma}
\begin{proof}
  Suppose we have two $\boundary$--compressing disks $D_1$ and $D_2$
  which are not isotopic but which meet $P$ in arcs of the same
  isotopy class.  Among all such disks, choose $D_1$ and $D_2$ to
  minimize $|D_1\cap D_2|$.  If $D_1\cap D_2=\emptyset$, then after
  cutting $H$ along $D_1$ we can isotop $\boundary D_2$ to lie in
  $\boundary H\setminus P$.  This surface is incompressible, so $D_2$
  must be trivial here.  This means that $D_2$ is parallel to $D_1$ in
  $H$.

  So assume that $D_1\cap D_2\neq\emptyset$.  Since $\boundary D_1\cap
  P$ and $\boundary D_2\cap P$ are of the same class in $P$, we may
  isotop $D_1$ and $D_2$ so that $\boundary D_1\cap\boundary D_2\cap
  P$ is empty.  This isotopy does not introduce new intersections, so
  consider an arc $\alpha$ of $D_1\cap D_2$ which is outermost in
  $D_1$.  This arc cuts off a subdisk $\delta$ of $D_1$ whose interior
  does not meet $D_2$.  We may surger $D_2$ along $\delta$ to obtain a
  new $\boundary$--compressing disk $D_2'$ with $D_2\cap P=D_2'\cap
  P$.  This disk is disjoint from $D_2$ and has at least one fewer arc
  of intersection with $D_1$, and so both pairs $(D_2, D_2')$ and
  $(D_1, D_2')$ are parallel.  This means that $D_1$ and $D_2$ are
  also parallel.
\end{proof}

\begin{lemma}\label{boundary-compressing-disks-meet-K}
  Every $\boundary$--compressing disk for $P$ in $H$ meets $K$ at 
  least once.
\end{lemma}
\begin{proof}
  Let $D$ be an essential disk in $H$ which does not meet $K$.  If $D$
  is a $\boundary$--compressing disk for $P$, then
  by~\autoref*{E-is-unique} and~\autoref*{minimum-of-E-cap-P} it must
  be separating.  Suppose that $\boundary D$ meets $\boundary_i P$.
  After compressing along $D$ we obtain two solid tori $S_1$ and
  $S_2$.  Let $D_j$ be the nonseparating $\boundary$--compressing disk
  for $P$ which does not meet $\boundary_jP$.  This disk is unique up
  to isotopy by the previous lemma.  By~\autoref*{attaching-2-h} each
  of these disks meets $K$ at least once.  But $D_x$ and $D_y$ are
  meridian disks for $S_1$ and $S_2$, where $\set{i, x,
    y}=\set{1,2,3}$.  This is impossible since $K$ is contained in
  either $S_1$ or $S_2$.
\end{proof}

\begin{prop}\label{K-atoroidal}
  Both the spaces $\closure{H\setminus N(K)}$ and $\closure{M\setminus
    N(K)}$ are irreducible and atoroidal.  The boundary of $N(K)$ is
  incompressible in $\closure{H\setminus N(K)}$ and $\closure{M\setminus
    N(K)}$.  There are no essential annuli in $\closure{M\setminus N(K)}$.
\end{prop}
\begin{proof}
  Let $F$ be a sphere in $H$ which does not bound a ball in the
  complement of $K$.  Then either $F$ does not bound a ball in $H$
  (impossible since handlebodies are irreducible) or $F$ bounds a ball
  containing $K$.  However, there are meridian disks in $H$ which have
  nonzero algebraic intersection with $K$.


  Suppose that $T\subseteq \closure{M\setminus N(K)}$ is an essential
  torus with $|T\cap P|$ minimal.  If $|T\cap P|>0$, examine the
  intersections $T\cap P$.  By the minimality of $|T\cap P|$ and
  incompressibility of $T$ and $P$, we may assume that there are no
  \sccs of intersection of $T\cap P$ which are trivial in $T$ or $P$.
  So suppose that there is a curve of intersection of $T\cap P$ which
  is essential in both $P$ and $T$.  Since $P$ is separating and $T$
  contains no trivial curves of intersection, we can find an annulus
  $A\subseteq T$ such that $\interior{A}$ has no intersections with
  $P$.  By~\autoref*{no-annuli-on-P}, we may isotop $T$ to reduce
  $|T\cap P|$.  This contradiction shows that $T\cap P=\emptyset$ and
  therefore $T$ lies in $H$.  However, this is impossible since $K$ is
  primitive in $H$.

  If $\boundary N(K)$ is compressible in $\closure{H\setminus N(K)}$,
  we see that either $K$ is contained in a ball or $H$ is reducible
  after compressing.  The same holds for $\closure{M\setminus N(K)}$.
  This contradicts the irreducibility of these spaces.

  Finally, suppose that $A$ is an essential annulus in
  $\closure{M\setminus N(K)}$ chosen to minimize $|A\cap P|$.  Using
  an innermost curve/outermost arc argument
  and~\autoref*{no-annuli-on-P} we can show that $A\cap P=\emptyset$.
  Suppose then that $A$ has one boundary component, $\boundary_1 A$,
  in $R$ and the other, $\boundary_2 A$, in $\boundary N(K)$.  Since
  $K$ does not bound a disk in $H$, The component $\boundary_2A$
  cannot be meridional on $\boundary N(K)$ because $A$ would become a
  compressing disk for $R$ under the trivial surgery.  Therefore
  $\boundary_2A$ meets the meridian of $\boundary N(K)$ at least once,
  and $\boundary_1A$ is isotopic to a component $\gamma$ of $\boundary
  P$.

  If $\boundary_2A$ meets a meridian of $\boundary N(K)$ more than
  once, then $\gamma$ is conjugate to an $n$-th power of $K$ in
  $\pi_1(H)$ for $n>1$.  Since $\gamma$ and $K$ are both primitive in
  $\pi_1(H)$, this is impossible.  Therefore $\boundary_2A$ is
  longitudinal on $\boundary N(K)$, and $A$ describes an isotopy of
  $K$ to $\gamma\subseteq R$.  This is impossible
  by~\autoref*{boundary-compressing-disks-meet-K}.

  If $A$ has both components on $\boundary M$, then $A$ lies in $J$ or
  $H$ disjoint from $P$.  But then $A$ is either an essential annulus
  in a hyperbolic knot exterior or an essential annulus in a
  compression body which is not vertical.  Both are impossible.

\end{proof}

\begin{lemma}\label{min-separating-disks}
  The boundary of every essential disk in $\closure{H\setminus N(K)}$
  meets $P$ in at least three arcs. 
\end{lemma}
\begin{proof}
  Let $D$ be a meridian disk for $H$ which does not meet $K$.
  By~\autoref*{boundary-compressing-disks-meet-K} and the
  incompressibility of $P$ in $H$, $D$ meets $P$ in at
  least two arcs.  If $D$ meets $P$ fewer than four times, it must be
  separating by~\autoref*{E-is-unique}
  and~\autoref*{minimum-of-E-cap-P}.  So let $D$ be a separating
  meridian disk for $H$ missing $K$ and meeting $P$ in two essential
  arcs $\lambda_1$ and $\lambda_2$.  The disk $D$ cuts $H$ into two
  solid tori $T$ and $T'$ with $K\subseteq T$.

  There are two arcs $\kappa_1$ and $\kappa_2$ parallel to $\lambda_1$
  so that each component of $P\setminus (\kappa_1\cup\kappa_2)$
  contains exactly one of $\lambda_1$ or $\lambda_2$.  Let $F_1$ and
  $F_2$ denote the product disks $\kappa_1\times I, \kappa_2\times
  I\subseteq P\times I$, and choose $D$ to minimize $|D\cap (F_1\cup
  F_2)|$ with the constraint that $\boundary D\cap P$ consists of the arcs
  $\lambda_1$ and $\lambda_2$.  Isotop $K$ to minimize $|K\cap(F_1\cup
  F_2)|$ while keeping $K\cap D=\emptyset$.  Since $F_1\cup F_2$
  separates $H$, $D\cap F_i$ is nonempty for at least one $i=1,2$.
  Without loss of generality assume $D\cap F_1\neq\emptyset$.

  We may assume that there are no \sccs of intersection in $D\cap
  F_1$.  Let $\alpha$ be an arc of $D\cap F_1$, outermost in $F_1$,
  cutting off a subdisk $F_1'$ of $F_1$ whose interior does not meet
  $D$.  If $\alpha$ cuts off a subdisk $\delta$ of $D$ which does not
  meet $P$, we could surger $D$ along $\delta$ to obtain a new disk
  meeting $P$ in the arcs $\lambda_1$ and $\lambda_2$ and having fewer
  intersections with $F_1$ and $F_2$.  Therefore $\alpha$ divides $D$
  into two halves each containing one of $\lambda_1$ or $\lambda_2$.
  Surgering along $F_1'$ we obtain two parallel
  $\boundary$--compressing disks for $P$ in $H$, and so we have shown
  that $D$ consists of $D_1$ and $D_2$, isotopic to product disks,
  tubed along an arc in $\boundary H\setminus P$.

  Reversing the surgery along $F_1'$ corresponds to adding a
  $1$--handle to one of the components of $H$ cut along $D_1\cup D_2$,
  and therefore $F_1'$ is nonseparating in its solid torus.  We claim
  that $F_1'$ is a meridian disk for $T$, the solid torus containing
  $K$.  If not, $F_1'\cap K=\emptyset$, and so $D_1$ and $D_2$ are
  $\boundary$--compressing disks for $P$ which do not meet $K$,
  impossible by~\autoref*{boundary-compressing-disks-meet-K}.  The
  exterior of $K$ in $H$ is irreducible and atoroidal
  by~\autoref*{K-atoroidal}, so $K$ is isotopic to a core of $T$ and
  therefore isotopic into $\boundary T\cap(\boundary H\setminus P)$.
  However, $K$ is clearly not isotopic in $H$ to a boundary component
  of $\boundary H\setminus P$, and so no such disks exist.
\end{proof}

\section{Surgery on the knots}
Fix a nontrivial $2$--bridge knot $L$ and two relatively prime
positive integers $p<q$, and let $(K,M) = (K^L_{p,q}, M^L_{p,q})$.  In
this section we show that the knot $K\subseteq M$ has exactly one
nontrivial surgery yielding a handlebody.

\begin{prop}\label{k-has-handlebody-surgery}
  The knot $K\subseteq M$ has a nontrivial handlebody surgery.
\end{prop}
\begin{proof}
  We will first show that $K\subseteq H$ has a handlebody surgery
  under which $P$ becomes $\boundary$--compressible.  To see this,
  push $K$ inside $H$ so that it lies in the punctured torus
  $(T^2\times\set{1/2})\cap H$.  Let $E(K)$ be the exterior of $K$ in
  $H$, and let $S$ be the $3$--punctured sphere
  $(T^2\times\set{1/2})\cap E(K)$.  This surface defines a slope on
  $\boundary N(K)$.  Since $K$ is isotopic to a primitive curve on
  $\boundary H$, surgery at this slope yields a handlebody.  After
  surgery, $S$ becomes a disk meeting $P$ in a single essential
  separating arc.

  Perform the $\boundary$--compression and glue the resulting two
  solid tori to $J$.  The cores of the gluing annuli are both
  primitive in $J$ by construction, so the result is a handlebody.
  Reversing the $\boundary$--compression does not change this.
\end{proof}

Let $\mu$ be the meridional slope on $\boundary N(K)$, and let
$\lambda$ be the slope on $\boundary N(K)$ given
by~\autoref*{k-has-handlebody-surgery}.  Let $H(\alpha)$ denote the
space resulting from Dehn surgery on $K$ in $H$ along the $\alpha$
slope.  Since $K$ is isotopic to a primitive curve in $\boundary H$,
$H(\alpha)$ is a handlebody for every slope $\alpha$.

\begin{prop}\label{2-handle-dual-knot}
  After attaching a $2$--handle to $H(\lambda)$ along $\boundary_3P$,
  we obtain a Seifert fiber space over the disk with two exceptional
  fibers of order $p$ and $q$.
\end{prop}
\begin{proof}
  Recall from~\autoref*{k-has-handlebody-surgery} that there is a
  separating $\boundary$--compressing disk $D$ for $P$ in
  $H(\lambda)$.  This disk divides $H(\lambda)$ into two solid tori:
  $T_1$ contains $\boundary_1P$ and $T_2$ contains
  $\boundary_2P$.  Attaching a $2$--handle to $H$ along
  $\boundary_1P$, we obtain a knot in a solid torus isotopic to a
  $(p,q)$ curve with respect to $m$ and $l$.  The
  surgery slope $\lambda$ is the surface slope with respect to this
  embedding, and so surgery yields the connect sum of a solid torus
  and a lens space of order $p$.  Similarly, attaching a $2$--handle
  to $H(\lambda)$ along $\boundary_2P$ yields the connect sum of a
  solid torus and a lens space of order $q$.  

  Let $D_1$ and $D_2$ be meridian disks for $T_1$ and $T_2$ disjoint
  from $D$.  Clearly $\algint{\boundary_1P}{\boundary D_1} = p$ and
  $\algint{\boundary_2P}{\boundary D_2}=q$, and so $H(\lambda)$
  appears as in~\autoref*{fig:H-dual}.  In this figure,
  $\boundary_1P$ runs $p$ times around the left handle and
  $\boundary_2P$ runs $q$ times around the right handle. With respect
  to the basis for $\pi_1(H(\lambda))$ dual to $\set{D_1,D_2}$, the
  fundamental group of $H(\lambda)[\boundary_3P]$ is
  $\presentation{x,y}{x^p=y^q}$.  This space is a Seifert fiber space
  over the disk with exceptional fibers of order $p$ and
  $q$~\cite{Stevens}.
\end{proof}

\begin{figure}[h!tb]
  \begin{center}
    \def\svgwidth{0.7\textwidth}
    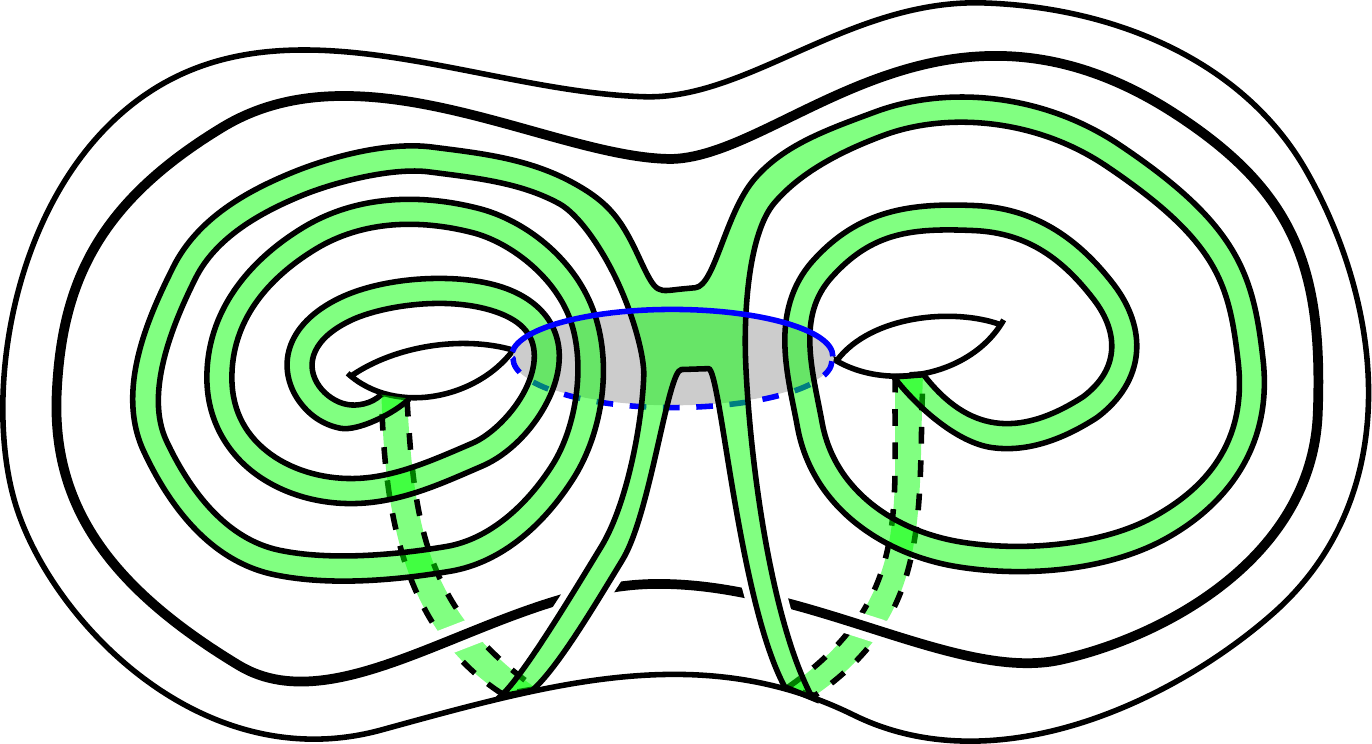
  \end{center}
  \caption{$H(\lambda)$, $K'$, $P$, and $E$}
  \label{fig:H-dual}
\end{figure}

After surgery on $K$, the core of the attached solid torus becomes a
new knot in $H(\lambda)$ which we will call the \defn{dual knot}.  Let
$K'\subseteq H(\lambda)$ be the dual knot to $K$.  In light
of~\autoref*{2-handle-dual-knot}, we get a picture of $H(\lambda)$ as
in~\autoref*{fig:H-dual}.  Notice in particular the disk $E$ and the 
intersections of its boundary with $P$.

\begin{corollary}\label{not-isotopic-to-boundary}
  The knot $K$ is not isotopic to the boundary of the solid torus
  $H[\boundary_3P]$. 
\end{corollary}
\begin{proof}
  Let $N=H[\boundary_3P]$ and suppose that $K$ is isotopic to
  $\boundary N$.  Recall that $K$ meets a meridian disk for $N$
  algebraically in $q-p$ points by~\autoref*{attaching-2-h}.
  Therefore $N(\lambda)$ is either a Seifert fiber space over the disk
  with two exceptional fibers of order $q-p$ and $m$ for some nonzero
  integer $m$ or (if $\lambda$ is the surface slope of $K$ in
  $\boundary N$) a connect sum of a solid torus and lens space.  The
  previous result shows that $N(\lambda)$ is a Seifert fiber space
  over the disk with exceptional fibers of orders $p$ and $q$.  This
  space is irreducible, so it cannot be a connect sum.  Furthermore,
  the classification of Seifert fiber spaces (see for
  example~\cite{Jaco}) then shows that $q-p=p$ or $q-p=q$, both of
  which are impossible.
\end{proof}

The following lemma is well known and can be proven using an innermost
curve/outermost arc argument. 

\begin{lemma}\label{sfce-in-hbody}
  Let $S\subseteq M$ be a properly embedded incompressible surface in
  a handlebody of genus $g>0$.  Then $S$ is $\boundary$--compressible.
\end{lemma}


The next proposition plays an important role in our argument that
many of the knots constructed above are not $1$--bridge.  In this
section we will use it to consider the effect of surgery on $K$ at
slopes other than $\mu$ and $\lambda$.

\begin{prop}\label{three-punctured-spheres}
  Let $P\subseteq M$ be a properly embedded, separating,
  incompressible 3--punctured sphere in a handlebody of genus $g>1$,
  so that $\boundary$--compressing $P$ in $M$ yields a separating
  surface $\mathcal{A}$ consisting of one or two annuli.  Then the
  core(s) of $\mathcal{A}$ are primitive on at least one side of $P$.
\end{prop}
\begin{proof}
  Suppose that $P$ separates $M$ into $M_1$ and $M_2$.  These are
  easily seen to be handlebodies themselves.  Suppose further that $D$
  is a $\boundary$--compressing disk for $P$ which is properly
  embedded in $M_2$. Let $\mathcal{A}$ be the annulus or annuli
  resulting from $\boundary$--compressing.  Then $\mathcal{A}$
  separates $M$ into $N_1$ and $N_2$ where $N_1\subseteq M_1$ is
  homeomorphic to $M_1$ and $N_2$ is either a solid torus or a pair of
  solid tori.  In either case, $M_2$ is obtained from $N_2$ by
  attaching a $1$--handle.

  Let $A$ be a component of $\mathcal{A}$.  Clearly $A$ is a properly
  embedded, incompressible annulus in $M$, and therefore it
  $\boundary$--compresses by~\autoref*{sfce-in-hbody}.  Let $D'$ be a
  $\boundary$--compressing disk for $A$ in $M$.  Using an innermost
  curve/outermost arc argument, we may assume that $D'$ lies entirely
  in $N_1$ or $N_2$.  If $D'\subseteq N_1$, then $D'$ shows that the
  core of $A$ is primitive in $M_1$ since $M_1\cong N_1$.  If
  $D'\subseteq N_2$, then the core of $A$ is primitive in $M_2$ since
  we may choose the attaching disk(s) of the $1$--handle to be
  disjoint from a disk showing that the core of $A$ is primitive.
\end{proof}

\begin{lemma}\label{primitive-curves-product}
  Let $P$ be a $3$--punctured sphere embedded in the boundary of a
  genus 2 handlebody $H$ so that no two components of $\boundary P$
  are homologous in $H$ and at least two components of $\boundary P$ are
  primitive in $H$.  If $P$ is $\boundary$--compressible, then either
  \begin{itemize}
    \item $H\cong P\times I$, or
    \item there is one component $z$ of $\boundary P$ which is not
      primitive, and every $\boundary$--compressing disk for $P$ is
      disjoint from $z$.
  \end{itemize}
\end{lemma}
\begin{proof}
  First note that if $D$ is nonseparating in $H$ but $\boundary D$
  separates $P$, two components of $\boundary P$ are homologous in the
  solid torus resulting from $\boundary$--compressing.  These
  components must have been homologous in $H$, and therefore $D$ meets
  two different components of $\boundary P$.

  Suppose that there is a component $z$ of $\boundary P$ which is not
  primitive in $H$ and let $D$ be $\boundary$--compressing disk for
  $P$ in $H$.  In this case $D$ must be disjoint from $z$ since
  otherwise it would mark $z$ as primitive.

  Therefore suppose that $D$ is separating in $H$ and meets $z$.
  Cutting $H$ along $D$ we obtain two solid tori each with an annulus
  embedded in its boundary.  The primitive components $x$ and $y$ of
  $\boundary P$ are disjoint from $D$, and so an innermost
  curve/outermost arc argument shows that they remain primitive in
  their respective solid tori.  Therefore these solid tori are
  products of the annuli in their boundaries.  Reversing the
  $\boundary$--compression preserves the product structure, so $H\cong
  P\times I$.

  The same argument shows that if every component of $\boundary P$ is
  primitive and $D$ is a separating $\boundary$--compressing disk for
  $P$ in $H$, then $H\cong P\times I$.

  If $D$ is nonseparating, a similar argument shows that we obtain a
  single solid torus with an annulus in its boundary whose core is
  primitive.  Again, we see that $H\cong P\times I$.
\end{proof}

\begin{prop}\label{two-handlebody-fillings}
  If surgery on $K$ in $M$ along the slope $\alpha$  yields a
  handlebody, then $\alpha=\mu$ or $\alpha=\lambda$.
\end{prop}
\begin{proof}
  Let $\alpha\neq\mu,\lambda$ be a slope on $\boundary N(K)$ such that
  surgery on $K$ in $M$ along $\alpha$ yields a handlebody.
  By~\autoref*{K-atoroidal} and a result of Wu~\cite{Wu92},
  $\Delta(\alpha,\mu)=\Delta(\alpha,\lambda)=1$.  Recall that
  $\boundary H$ contains two punctured tori
  $T_1=(T^2\times\set{0})\cap H$ and $T_2=(T^2\times\set{1})\cap
  H$. Isotop $K$ to lie in $T_1$; the surface slope of $K$ with
  respect to $\boundary H$ is $\lambda$.  Note that $K$ is disjoint
  from $\boundary_2P$.  The effect of $\alpha$ surgery on $\boundary
  H$ can be seen by Dehn twisting along $K$.  Since $K$ is disjoint
  from $\boundary_2P$, this curve remains primitive in $H(\alpha)$.
  Similarly, we may isotop $K$ to lie in $T_2$ to see that
  $\boundary_1P$ remains primitive in $H(\alpha)$.

  By a result of Ni~\cite{Ni}, if $H(\alpha)\cong P\times I$ then $K$
  meets every nonseparating $\boundary$--compressing disk for $P$ in
  $H$ at most twice.  Since this is clearly not
  true,~\autoref*{primitive-curves-product} implies that
  $\boundary_3P$ is not primitive in $H(\alpha)$ and that every
  $\boundary$--compressing disk for $P$ in $H(\alpha)$ is disjoint
  from $\boundary_3P$. 

  By~\autoref*{three-punctured-spheres}, $P$ must $\boundary$--compress in
  $M(\alpha)$, and the core(s) of the resulting annulus or annuli must
  be primitive on at least one side of $P$.  Since $\boundary_3P$ is
  not primitive in either $J$ or $H(\alpha)$, a
  $\boundary$--compressing disk $D\subseteq H(\alpha)$ for $P$ must
  meet $\boundary_3P$.  This
  contradicts~\autoref*{primitive-curves-product} and shows that
  $M(\alpha)$ is not a handlebody.
\end{proof}

\section{Infinitely many distinct examples}\label{section:distinct}
Here we show that infinitely many of the knots described above are
inequivalent in the sense that there is no homeomorphism of the
handlebody which carries the first knot to the second.  We need a
technical lemma which says roughly that the isotopy class of the
multicurve $\boundary P$ is determined by the construction
in~\autoref*{section:knots}:

\begin{lemma}\label{boundary-P-unique}
  Fix a hyperbolic $2$--bridge knot $L$ and integers $1<p<q$ with
  $\gcd(p,q)=1$.  Suppose that $(M, K)$ and $(M', K')$ arise from the
  construction in~\autoref*{section:knots}, and let $P$ and $P'$ be
  the corresponding $3$--punctured spheres described in that section.
  If $(M', K')$ is homeomorphic to $(M, K)$ by a homeomorphism $h$,
  then $h(\boundary P')$ is isotopic to $\boundary P$.
\end{lemma}
\begin{proof}
  We will work in $M$ and identify $P'$ with its image under $h$.
  Note that the roles of $P$ and $P'$ are symmetric in the proof.
  
  Isotop $P'$ so that $|P\cap P'|$ is minimal.  The curves $\boundary
  P$ cut $\boundary M$ into two $3$--punctured spheres, so if $P\cap
  P'=\emptyset$, then $\boundary P'$ is isotopic to $\boundary P$.

  If $P\cap P'\neq\emptyset$, there are several cases:
  \begin{enumerate}
  \item Suppose that $P\cap P'$ contains a \scc which is trivial in
    $P'$.  Then it must be trivial in $P$ because both surfaces are
    incompressible, and so we can isotop to reduce $|P\cap P'|$.

    If $P\cap P'$ contains an arc which is trivial in $P'$, then an
    outermost such arc gives either a $\boundary$--compressing disk
    for $P$ in the complement of $K$ (impossible by~\autoref*{P-in-J}
    and~\autoref*{boundary-compressing-disks-meet-K}) or a
    disk which guides an isotopy of $P'$ reducing $|P\cap P'|$.

  \item If there is a \scc of intersection of $P\cap P'$ which is
    essential in $P'$ then it must also be essential in $P$.  In this
    case there is an incompressible annulus $A\subseteq N$, $N=J$ or
    $H$, with one boundary component, $\boundary_1A$, on $P$, and the
    other, $\boundary_2A$, on $\boundary N\setminus P$.  We may isotop
    $\boundary_2A$ to lie on $P$ and use~\autoref*{no-annuli-on-P} to
    reduce $|P\cap P'|$.

  \item If $P\cap P'$ consists of a single arc which separates $P'$,
    then there is a separating, incompressible annulus $A$ properly
    embedded in $H$ with one boundary component, $\boundary_1 A$,
    lying in $\boundary H\setminus P$ and the other, $\boundary_2 A$,
    meeting $P$ in a single separating arc.  Let $\rho$, $\sigma,$ and
    $\tau$ be the components of $\boundary P$ where  $\boundary_2
    A\cap\rho\neq\emptyset$ and $\boundary_1 A$ is parallel to
    $\sigma$.

    Recall that $\boundary$--compressing disks for $P$ in $H$ are
    product disks.  There is a $\boundary$--compressing disk
    $D_{\rho\sigma}$ for $P$ meeting $\rho$ and $\sigma$, and another
    one, $D_{\rho\tau}$, meeting $\rho$ and $\tau$.  Furthermore, we
    may choose these disks to be disjoint from $\boundary_2A$ in $P$.
    Since $\boundary_2 A$ is homotopic to $\sigma$ in $H$,
    $\algint{\boundary_2A}{D_{\sigma\rho}}=1$ and $\algint{\boundary_2
      A}{D_{\tau\rho}} = 0$.  However, there are no essential arcs in
    $\boundary H\setminus P$ which connect $\rho$ to itself and have
    these intersection numbers.  Therefore we may isotop $\boundary_2
    A$ to lie entirely in $P$, and then we are in a case already dealt
    with.

  \item If there is a pair of arcs of intersection of $P\cap P'$ which
    are parallel in $P'$, let $D$ be a subdisk of $P'$ which is
    bounded by two such arcs whose interior is disjoint from $P$.  If
    $D\subseteq H$, then $D$ must be inessential
    by~\autoref*{min-separating-disks}, so we can reduce $|P\cap P'|$.
    Therefore suppose that $D\subseteq J$ and that $D$ is an $(r, s)$
    disk in the sense that it intersects
    $\boundary_1P\cup\boundary_2P$ in $r$ points and $\boundary_3P$ in
    $s$ points (\cf proof of~\autoref*{P-in-J}).  Then $r+s = 4$.
    Suppose further that $(J, P)$ arises as the exterior of an $m/n$
    rational tangle as in~\autoref*{section:knots}.  Let $\inv{m}$ be
    the mod $n$ inverse of $-m$ such that $2|\inv{m}|\leq n$.  Since
    $L$ is hyperbolic, $|\inv{m}|>1$.  Note also that $r$ is even
    because $D\cap P$ consists of two parallel arcs.  Therefore
    by~\cite[Lemma 3.4]{Wu12}, $s\geq 4$, and so $D$ is a $(0,4)$
    disk.  However, this is impossible by~\cite[Lemma 2.3(1)]{Wu12}.

  \item If there are no parallel arcs of intersection of $P\cap P'$ in
    $P$ or $P'$, then $P\cap P'$ consists of three mutually
    nonparallel arcs in both $P$ and $P'$.  We may choose a subdisk
    $D$ of $P'$ bounded by these arcs such that $D\subseteq J$.  In
    this case, $D$ is a $(4, 2)$ disk.  However, from the previous
    argument we know that $s\geq 4$, a contradiction.
  \end{enumerate}
\end{proof}

\begin{prop}\label{infly-many-examples}
  Let $L_1$ and $L_2$ be hyperbolic $2$--bridge knots which are not
  mirror images, and let $p,q,p'$, and $q'$ be integers with $p<q$,
  $p'<q'$, $\gcd(p,q)=1$, and $\gcd(p',q')=1$.  Then there is a
  homeomorphism $M^{L_1}_{p,q}\to M^{L_2}_{p',q'}$  taking
  $K^{L_1}_{p,q}$ to $K^{L_2}_{p',q'}$ iff $L_1=L_2$, $p=p'$, and
  $q=q'$.
\end{prop}
\begin{proof}
  When $L$ is a hyperbolic 2--bridge knot, the triple of spaces
  obtained by attaching a $2$--handle along each component of
  $\boundary P^L$ is an invariant of $(M^L, K^L)$
  by~\autoref*{boundary-P-unique}.  After attaching a $2$--handle
  along $\boundary_3 P^{L_i}\subseteq M^{L_i}$, we obtain $E(L_i)$,
  $i=1,2$.  Since the other two components of $\boundary P^{L_i}$ are
  primitive, attaching a 2--handle along them yields a solid torus.

  First, focus on the 2--handle whose attaching curve is $\boundary_3
  P^{L_i}$.  A homeomorphism of $(M^{L_1}, K^{L_1})$ with $(M^{L_1},
  K^{L_2})$ would extend across this $2$--handle to give a
  homeomorphism of $E(L_1)$ with $E(L_2)$.  However, these spaces are
  not homeomorphic unless $L_1=L_2$~\cite{knots-are-determined}.

  Similarly, a homeomorphism of $(M^{L}_{p,q}, K^L_{p,q})$ with
  $(M^L_{p',q'}, K^L_{p',q'})$ would extend across a $2$--handle
  attached along $\boundary_1P$ or $\boundary_2 P$ to give 
  homeomorphisms $h_1$ and $h_2$ of solid tori such that
  $h_i(K_{p,q})=K_{p',q'}$, $i=1,2$.  By~\autoref*{attaching-2-h}, we
  must have $p=p'$ and $q=q'$.
\end{proof}

\begin{prop}
  The spaces $(M(\lambda), K')$ do not arise from the construction
  of~\autoref*{section:knots}.
\end{prop}
\begin{proof}
  From the proof of~\autoref*{infly-many-examples}, we know that the
  triple of spaces obtained by attaching $2$--handles along the
  components of $\boundary P$ is an invariant of the spaces
  constructed in~\autoref*{section:knots}.  Furthermore, this
  invariant consists of a $2$--bridge knot exterior and two solid
  tori.  But from the proof of~\autoref*{2-handle-dual-knot} we see
  that attaching $2$--handles to $M(\lambda)$ along $\boundary_1 P$
  yields the connect sum of a lens space of order $p$ and a solid
  torus.
\end{proof}

\section{\texorpdfstring{The position of tunnels}{The position of tunnels}}
Fix a nontrivial $2$--bridge knot $L$ and integers $0<p<q$ with
$\gcd(p,q)=1$, and let $(M,K)=(M^L_{p,q},K^L_{p,q})$.  Suppose that
$K$ is 1--bridge in $M$ so that there is a tunnel $t$ for $K$ as
in~\autoref*{1-bridge-iff-1-tunnel}.  Then $M'=\closure{M\setminus
  N(K\cup t)}$ is a handlebody of genus three.  In this section we
will isotop $t$ to lie in the handlebody $H$ while keeping
$K\subseteq H$.  Then we will isotop $t$ to be disjoint from the disk
$E$ described in~\autoref*{chapter:knots}.  This will allow us to make
arguments about $\boundary$--compressing disks for $P$ in $M'$.

\subsection{\texorpdfstring{The position of $t$ with respect to
    $P$}{The position of t with respect to P}}
First suppose that we have isotoped $t$ to minimize the number of
intersections of $P$ and $t$ while keeping $K\subseteq H$.  Define
$P'=P\cap M'$ and let $M'=H'\cup_{P'} J'$ where $H'\subseteq H$ and
$J'\subseteq J$.  Note that $P'$ is incompressible in $M'$ since if it
were not, we could find a compressing disk and use it to reduce
$|P\cap t|$.

Let $D$ be a meridian disk for $M'$, suppose that $|P\cap t|> 0$, and
suppose further that $D$ does not intersect $P'$.  Then we may isotop
$D$ so that $\boundary D\cap \boundary N(t)$ is empty, and therefore
$\boundary D$ lies either on $\boundary N(K)$, $\boundary J\setminus
P$, or $\boundary H\setminus P$.  The first cannot happen
by~\autoref*{K-atoroidal}, and the second two are impossible because
both of these surfaces are incompressible.  Therefore either $t\cap P$
is empty or $D$ meets $P'$.

\begin{lemma}\label{t-cap-P-is-empty}
  The intersection $t\cap P$ is empty.
\end{lemma}
\begin{proof}
  We may isotop $D$ so that it intersects $P'$ minimally and $D\cap
  P'$ consists of arcs essential in $P'$.  An outermost arc in $D$
  gives a disk $D'$ with $\boundary D' = \alpha\cup\beta$, $\alpha\subseteq P'$
  essential and $\beta\subseteq \boundary M'$.

  Call components of $\boundary P'\setminus\boundary P$ \defn{new}
  boundary components and components of $\boundary P$ \defn{old}.  Then
  there are several cases:

  \begin{enumerate}
  \item The arc $\alpha$ connects two distinct old boundary
    components.  Then we get a boundary compressing disk for $P$,
    which must lie in $H$ by~\autoref*{P-in-J}.  The boundary
    compression is disjoint from $K$.  However, there are
    no $\boundary$--compressing disks for $P$ in $H$ that do not meet
    $K$ by~\autoref*{boundary-compressing-disks-meet-K}.

  \item The arc $\alpha$ connects an old boundary component to itself.
    Suppose that $\alpha$ is nontrivial in $P$.  Then we may isotop
    $\boundary D'$ so that $\boundary D'\cap \boundary
    N(t)=\emptyset$.  Therefore $D'$ is a $\boundary$--compressing
    disk for $P$ in $J$ or $H$, and since we know that $P$ is
    $\boundary$--incompressible in $J$, $D'$ must be a
    $\boundary$--compressing disk for $P$ in $H$.  
    However, $D'\cap K=\emptyset$,
    contradicting~\autoref*{boundary-compressing-disks-meet-K}.


    If $\alpha$ is trivial in $P$, then $\beta$ must be inessential in
    $\closure{\boundary J\setminus P}$ or $\closure{\boundary
      H\setminus P}$ because otherwise we would obtain a compressing
    disk for this incompressible surface.  But in this case we may
    isotop $D'$ to be a compressing disk for the incompressible
    surface $P'$.

  \item The arc $\alpha$ connects an old and a new boundary component.
    Then $D'$ guides an isotopy of $t$ reducing the number of
    intersections of $t$ with $P$ by one and leaving one endpoint of
    $t$ on the other side of $P$.

  \item The arc $\alpha$ connects two new boundary components.  Then
    $D'$ guides an isotopy of $t$ reducing the number of intersections
    of $t$ with $P$ by two.

  \item The arc $\alpha$ connects a new component to itself. If
    $\alpha$ extends to an inessential \scc in $P$, then $\beta$ must
    be essential in $\closure{\boundary H'\setminus P'}$ or
    $\closure{\boundary J'\setminus P'}$ because otherwise we would
    get a compressing disk for $P'$ by pushing $\beta$ to $P'$.  The
    disk $D'$ extends to an annulus in $H$ or $J$, and we can cap off
    the boundary component containing $\alpha$ in $P$ (since this
    curve is inessential there) to get a compressing disk for
    $\closure{\boundary J\setminus P}$, $\boundary N(K)$,
    or $\closure{\boundary H\setminus P}$.  These are impossible
    by~\autoref*{P-in-J}, \autoref*{K-atoroidal}, and the fact that
    $H\cong P\times I$, respectively.

    So suppose that $\alpha$ extends to an essential \scc in $P$.
    Then $D'$ extends to an annulus $A$ in $N=J$ or $H$ with one
    boundary component, $\boundary_1A$, on $P$, and the other,
    $\boundary_2A$, on $\boundary N$.  The surface $A$ is
    incompressible because $\alpha$ extends to an essential \scc in
    $P$.  By~\autoref*{K-atoroidal}, $\boundary_2A$ cannot lie on
    $\boundary N(K)$.  Therefore $\boundary_2A$ is parallel to a
    component of $\boundary P$ in $\boundary N$.  We may isotop $A$ so
    that $\boundary A\subseteq P$ and use~\autoref*{no-annuli-on-P} to
    isotop $A$ through $P$, reducing $|t\cap P|$.
  \end{enumerate}
\end{proof}

\subsection{\texorpdfstring{The position of $t$ with respect to
    $E$}{The position of t with respect to E}}
By~\autoref*{t-cap-P-is-empty}, we may take $t\cap P=\emptyset$.  Let
$E'$ be the planar surface $E\cap H'$.  Note that $P$ is
$\boundary$--compressible in $H'$ by~\autoref*{sfce-in-hbody}
and~\autoref*{P-in-J}.

\begin{lemma}\label{boundary-D-cap-P}
  Let $D$ be a $\boundary$--compressing disk for $P$ in $H'$ which
  minimizes the tuple $(|D\cap E'|, |D\cap\boundary N(t)|, |\boundary
  D\cap \boundary E'\cap P|)$.
  \begin{itemize}
    \item If $\boundary D$ does not separate $P$, then $\boundary
      D\cap\boundary E'\cap P=\emptyset$.
    \item If $\boundary D\cap P$ connects $\boundary_1P$ to itself,
      then $|\boundary D\cap\boundary E'\cap P|=q-1$.
    \item If $\boundary D\cap P$ connects $\boundary_2P$ to itself,
      then $|\boundary D\cap\boundary E'\cap P|=p-1$.
    \item If $\boundary D\cap P$ connects $\boundary_3P$ to itself, then
      $|\boundary D\cap\boundary E'\cap P|=1$.
  \end{itemize}
\end{lemma}
\begin{proof}
  Recall that the arcs of intersection of $E'$ and $P$ are
  nonseparating according to~\autoref*{minimum-of-E-cap-P}.  The first
  claim follows from the fact that we can isotop nonseparating arcs in
  a 3--punctured sphere to be disjoint.  The isotopy pushes
  intersections of arcs in $P$ to intersections of arcs in
  $\closure{\boundary H'\setminus P}$ and can be chosen so that it
  does not increase $|D\cap E'|$, $|\boundary D\cap P|$, or
  $|\boundary D\cap\boundary N(t)|$.

  An essential separating arc meets a nonseparating arc in a pair of
  pants minimally zero or one times.  Recall that $\boundary E\cap P$
  consists of a single arc connecting $\boundary_1P$ and
  $\boundary_2P$, $p-1$ arcs connecting $\boundary_1P$ and
  $\boundary_3P$, and $q-1$ arcs connecting $\boundary_2P$ and
  $\boundary_3P$.  Therefore we may isotop $D$ to one of the forms
  above.  This isotopy pushes intersections of arcs in $P$ to
  intersections of arcs in $\boundary H'\setminus P$ and can be chosen
  so that it does not increase $|D\cap E'|$, $|\boundary D\cap P|$, or
  $|\boundary D\cap\boundary N(t)|$.
\end{proof}

Let $D$ be a $\boundary$--compressing disk for $P$ in $H'$ which meets
$\boundary N(t)$.  Denote by $\epsilon$ the subarc of $\boundary D$
such that $\epsilon\cap P\neq\emptyset$,
$\boundary\epsilon\subseteq\boundary N(t)$, and
$\interior{\epsilon}\cap\boundary N(t)=\emptyset$.  This is a proper
subarc of $\boundary D$
by~\autoref*{boundary-compressing-disks-meet-K}.  It is the largest
subarc of $\boundary D$ lying in $\boundary H$ which contains
$P\cap\boundary D$.

\begin{lemma}\label{t-cap-E-is-empty}
  We can isotop $t$ so that $t\cap E=\emptyset$.  During the isotopy,
  one foot of $t$ lies in $\boundary H\setminus P$.
\end{lemma}
\begin{proof}
  Isotop $t$ to intersect $E$ minimally, recall that we defined
  $E'=E\cap H'$, and let $D$ be a $\boundary$--compressing disk for
  $P$ in $H'$ minimizing $(|D\cap E'|, |D\cap\boundary N(t)|,
  |\boundary D\cap \boundary E'\cap P|)$.  If $D\cap\boundary
  N(t)=\emptyset$, then $\boundary D$ lies on $\boundary N(K)$ or
  $\boundary H$.  It cannot lie on $\boundary N(K)$ since $D$ is a
  $\boundary$--compressing disk for $P$.  If $\boundary D\subseteq H$,
  then $D$ gives a disk $D'$ in $H$ which is a
  $\boundary$--compressing disk for $P$ but does not meet $K$.  This
  is impossible by~\autoref*{boundary-compressing-disks-meet-K}, and
  therefore $D\cap E'\neq\emptyset$.

  By a standard innermost circle argument we may assume that $D\cap
  E'$ consists of arcs; let $\alpha$ be such an arc outermost in $D$.  

  Suppose that $\alpha$ cuts off a subdisk $D'$ of $D$ with $\boundary
  D'=\alpha\cup\beta$ where $\beta\subseteq\boundary H$ and $\beta\cap
  P=\emptyset$.  Then $\alpha$ is a properly embedded arc in $E$, and
  we can surger $E$ along $D'$ to obtain two new disks $E_1$ and $E_2$
  in $H$.  If $\boundary E_i$ meets $P$, then by considering the
  algebraic intersection of $\boundary E_i$ and $\boundary P$ we see
  that $E_i$ is essential.  If $\boundary E_i$ is disjoint from
  $\boundary P$ then $E_i$ is trivial because $\boundary H\setminus P$
  is incompressible.  We can then isotop $D$ to reduce $D\cap E'$
  while keeping $D$ a $\boundary$--compressing disk for $P$.
  Therefore both $E_1$ and $E_2$ are essential in $H$.  One of these
  new disks, say $E_1$, must be nonseparating since $E$ is.
  By~\autoref*{E-is-unique}, $E_1$ is isotopic to $E$.  However,
  by~\autoref*{min-separating-disks}, $E_2$ meets $P$ in at least
  three arcs, and therefore $E_1$
  violates~\autoref*{minimum-of-E-cap-P}.

  Suppose that $\alpha$, seen as an arc in $D$, has both endpoints in
  the same subarc of $\boundary N(t)\cap\boundary D$.  Then $\alpha$
  cuts off a subdisk $D'$ of $D$ which guides an isotopy of $t$
  through $E$ to reduce $|t\cap E|$ by two.  Note that the foot of $t$
  on $\boundary H$ does not meet $P$ during this isotopy.

  Suppose that $\alpha$ has one endpoint in an arc of $\boundary
  N(t)\cap\boundary D$ and the other in an adjacent arc of $\boundary
  H\cap\boundary D$ so that it cuts of a subdisk $D'$ of $D$ which is
  disjoint from $P$.  Then $D'$ guides an isotopy of $t$ which reduces
  $|t\cap E|$ by one.  The foot of $t$ on $\boundary H$ does not meet
  $P$ during this isotopy.

  Suppose that $\alpha$ has one endpoint in an arc of $\boundary
  N(t)\cap\boundary D$ and the other endpoint in a different arc of
  $\boundary N(t)\cap\boundary D$ so that $\alpha$ cuts off a subdisk
  $D'$ of $D$ which does not meet $P$.  The disk $D'$ extends to an
  embedded annulus $A$ in $H$ with one boundary component on $E$ and
  the other on either $\boundary N(K)$ or $\boundary H$.  We may cap
  off the boundary component of $A$ on $E$ to obtain a disk $D''$ with
  either $\boundary D''\subseteq \boundary N(K)$ or $\boundary
  D''\subseteq \boundary H$.  Note that $\boundary D''$ must be
  essential on $\boundary H$ or $\boundary N(K)$ since otherwise we
  could reduce $|\boundary D\cap \boundary N(t)|$.  No essential curve
  on $\boundary N(K)$ bounds a disk in $\closure{H\setminus N(K)}$, so
  $\boundary D''$ must lie on $\boundary H$.  In this case, note that
  $\boundary D''\cap P=\emptyset$ and $D''\cap K=\emptyset$.  But we
  have already seen that disks which do not meet $K$ must meet $P$.
  Therefore we may rule out such arcs.

  Finally, note that since $t\cap E\neq\emptyset$, every component of
  $\boundary N(t)\cap\boundary D$ contains the endpoint of at least
  one arc of $D\cap E'$.  Therefore there must be an arc of $D\cap E'$
  which is outermost in $D$ and has one of the forms above unless
  every arc meeting $\boundary N(t)$ has its other endpoint in $P$.
  So suppose that this is the case, and recall that $\epsilon$ is the
  largest subarc of $\boundary D$ which lies in $\boundary H$ and
  meets $P$.  Since $\epsilon\cap\boundary E'\cap P\neq\emptyset$, $D$
  must separate $P$ by~\autoref*{boundary-D-cap-P}.  By the minimality
  of $\boundary D\cap \boundary E\cap P$, we know exactly where the
  endpoints of $\epsilon$ lie on $\boundary P$; the possibilities are
  shown in~\autoref*{fig:E-cap-D-separating} (cf.~\autoref*{fig:PR}). 

  \begin{figure}[h!tb]
    \begin{center}
      \def\svgwidth{0.6\textwidth}
      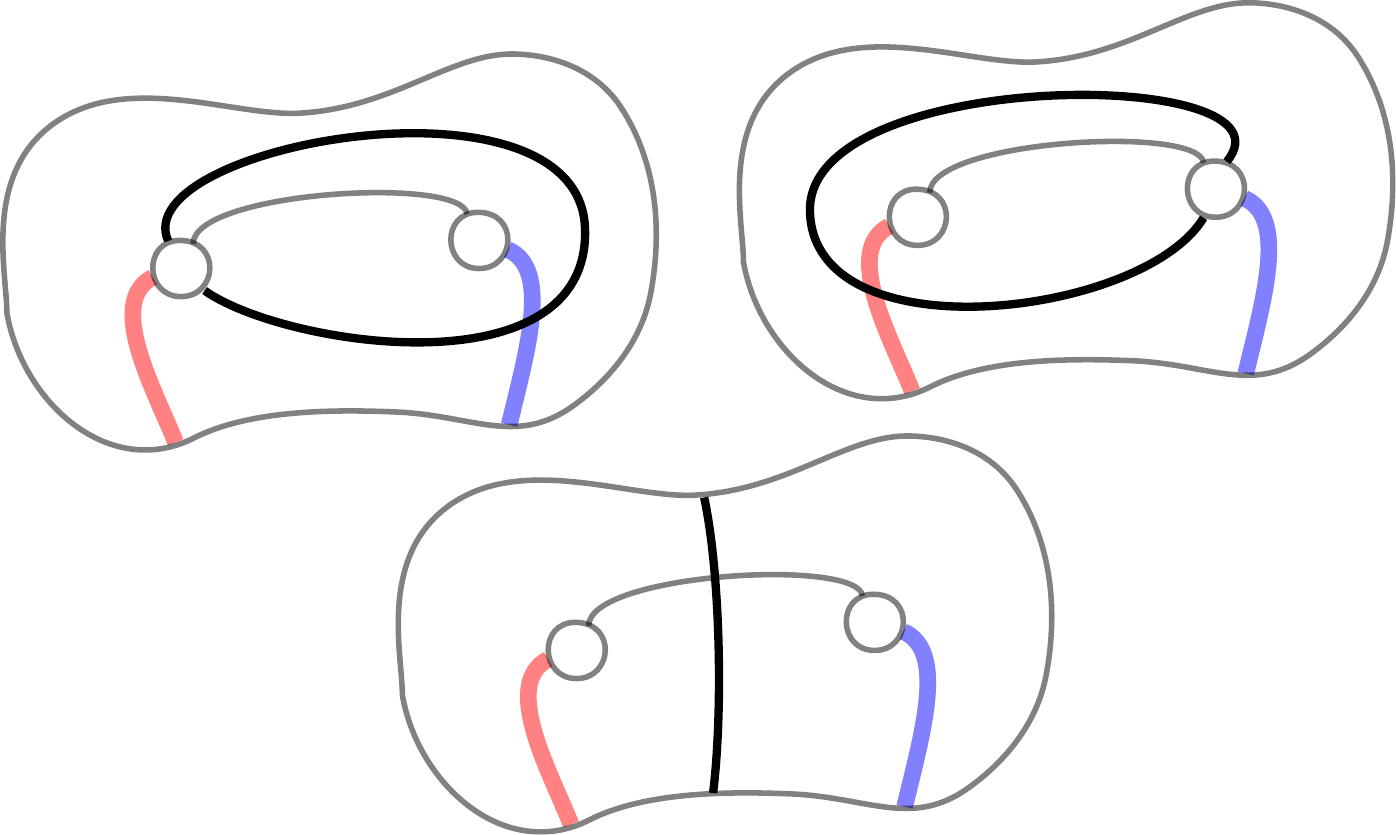
    \end{center}
    \caption{The possibilities for $\boundary D\cap P$ when $\boundary
      D$ is separating (bold)}
    \label{fig:E-cap-D-separating}
  \end{figure}
  
  Examining~\autoref*{fig:PR} we see that $\epsilon\cap R\cap\boundary
  E\neq\emptyset$.  But in this case there is an arc of intersection,
  outermost in $D$, which connects $P$ to an adjacent arc of
  $\boundary H$.  This arc cuts off a subdisk $D'$ of $D$ whose
  interior does not meet $E\subseteq H$.  We may surger $E$ along
  $D'$ to obtain two disks $E_1$ and $E_2$ in $H$.  By considering
  the algebraic intersection of $\boundary E_i$ and $\boundary P$ we
  see that $E_i$ is essential, $i=1,2$.  One of these, say $E_1$, must
  be nonseparating since $E$ is.  By~\autoref*{E-is-unique}, $E_1$ is
  isotopic to $E$.  However, by~\autoref*{min-separating-disks}, $E_2$
  meets $P$ in at least three arcs, and therefore $E_1$
  violates~\autoref*{minimum-of-E-cap-P}.
\end{proof}

\section{\texorpdfstring{Boundary compressing disks for $P$ in
    $H'$}{Boundary compressing disks for P in H'}}
In this section we assume the existence of a tunnel $t$ as
in~\autoref*{1-bridge-iff-1-tunnel} so that $M'=\closure{M\setminus
  N(K\cup t)}$ is a handlebody of genus three.  We may assume that
$H'=\closure{H\setminus N(K\cup t)}$ is a handlebody of genus 3
by~\autoref*{t-cap-P-is-empty}.  Because it is not
$\boundary$--compressible in $J$, the surface $P\subseteq M'$ is
$\boundary$--compressible in $H'$ by~\autoref*{sfce-in-hbody}.  We
want to show that $\boundary$--compressing along any
$\boundary$--compressing disk for $P$ yields an annulus whose core is
not primitive in either $H'$ or $J$,
contradicting~\autoref*{three-punctured-spheres} and showing that $K$
cannot be 1--bridge.

\autoref*{t-cap-E-is-empty} shows that the nonseparating disk $E$
which is disjoint from $K$ in $H$ is also an essential disk in $H'$.
We consider first the case when there is a $\boundary$--compressing
disk disjoint from $E$.

\subsection{\texorpdfstring{Boundary compressing disks disjoint from
    $E$}{Boundary compressing disks disjoint from E}}
Let $D$ be a $\boundary$--compressing disk for $P$ in $H'$ such that
$D\cap E=\emptyset$.  Choose $D$ to minimize $(|\boundary
D\cap\boundary N(t)|, |\boundary D\cap \boundary E\cap P|)$.  Note
that $D$ must be nonseparating in $P$ since a separating disk
necessarily meets $E$ in $P$.  If $D$ does not meet $\boundary N(K)$,
then we may isotop $\boundary D$ so that it lies entirely in
$\boundary H$.  Thus $D$ gives a disk in $H$ with $D\cap K$ empty and
$D\cap P$ consisting of one essential arc, which is impossible
by~\autoref*{boundary-compressing-disks-meet-K}.

\begin{lemma}\label{one-arc-on-K}
  Suppose that $D$ is a boundary compressing disk for $P$ in $H'$ such
  that $D\cap E=\emptyset$.  Then we may isotop $D$ so that $D\cap E$
  is still empty and $\boundary D\cap \boundary N(t)$ consists of two arcs.
\end{lemma}
\begin{proof}
  Suppose the endpoint of the tunnel is in region $R_t$ of
  $R\setminus\boundary E$.  Then $R_t\cap H'\cap\boundary D$ consists
  of a number of arcs, exactly two of which meet $\boundary R_t$.
  Recall that every component of $R\setminus\boundary E$ is a disk.
  Therefore if there are more than two arcs, there is a trivial arc of
  $R_t\cap H'\cap\boundary D$ which we can use to reduce $|\boundary
  D\cap\boundary N(t)|$.
\end{proof}

\begin{lemma}\label{core-not-prim}
  Let $D$ be a $\boundary$--compressing disk for $P$ disjoint from
  $E$.  Then the core of the annulus obtained by
  $\boundary$--compressing $P$ along $D$ is not primitive in either
  $J$ or $H'$.
\end{lemma}
\begin{proof}
  By~\autoref*{one-arc-on-K}, the boundary of any
  $\boundary$--compressing disk $D$ disjoint from $E$ must contain a
  single arc in $P$ and two arcs connecting $\boundary P$ to the
  tunnel.  It follows that there is a region of $P$ whose boundary
  meets some region of $R$ in two distinct subarcs.

  If a rectangular region of $P$ is attached along $\boundary P$ to a
  region of $R$, we see that $\boundary E$ is disconnected.  The same
  holds if a rectangular region of $R$ is attached along $\boundary P$
  to a region of $P$.  Therefore $\boundary D\cap P$ and $\boundary
  D\cap R$ lie in the hexagonal regions, and it follows from
  examining~\autoref*{fig:PR} that $\boundary
  D\cap\boundary_3P=\emptyset$.

  By~\autoref*{one-arc-on-K}, the disk $D$ extends to an annulus $A$
  embedded in $\closure{H\setminus N(K)}$ with one boundary component,
  $\boundary_1A$, on $\boundary H$, and the other, $\boundary_2A$, on
  $\boundary N(K)$.  Furthermore, $\boundary_1A$ meets $P$ in an arc
  connecting $\boundary_1P$ and $\boundary_2P$.  Let $\alpha$ be the
  core of the annulus obtained by $\boundary$--compressing $P$ in
  $H'$, so that $\alpha$ is parallel to $\boundary_3P$.  This curve is
  not primitive in $J$ by construction.


  Attaching a $2$--handle to $H$ along $\alpha$ we obtain a solid
  torus $S$ containing a knot $K$.  The space $\closure{S\setminus
    N(K)}$ is irreducible by~\autoref*{attaching-2-h}.  By~\cite[Lemma
  2.5.3]{CGLS}, $\boundary_2A$ is either meridional on $\boundary
  N(K)$ or meets a meridian of $\boundary N(K)$ exactly once.  The
  latter case is impossible since $K$ is not isotopic to $\boundary S$
  by~\autoref*{not-isotopic-to-boundary}.  Therefore $A$ extends to a
  meridian disk $D'$ of $S$ meeting $K$ exactly once, which shows that
  $K$ is the connect sum of the core curve of $S$ with a nontrivial
  knot.  Therefore the exterior of $K\cup t$ in $S$ is homeomorphic to
  a nontrivial knot exterior with a $1$--handle attached, and so
  $\alpha$ is not primitive in $H'$.
\end{proof}

\subsection{\texorpdfstring{Boundary compressing disks meeting
    $E$}{Boundary compressing disks meeting E}}
In this section, let $D$ be a $\boundary$--compressing disk for $P$ in
$H'$ which minimizes $(|D\cap E|, |D\cap\boundary N(t)|, |\boundary
D\cap \boundary E\cap P|)$ and suppose that $|D\cap E|>0$.  We may
assume that every component of $D\cap E$ is an arc.  There are several
types of ``forbidden'' arcs in $D$ and $E$ which we will use to argue
about the existence of $\boundary$--compressing disks for $P$ in $H'$:

\begin{lemma}\label{arcs-in-E}
  Let $\alpha$ be a component of $D\cap E$.  Then either
  \begin{itemize}
  \item $\alpha$ connects two components of $\boundary E\setminus P$
    and separates $E$ into two subdisks each containing at least two
    components of $\boundary E\cap P$ on their boundary, or
  \item $\alpha$ connects a component of $\boundary E\cap P$ with a
    nonadjacent component of $\boundary E\setminus P$.
  \end{itemize}
\end{lemma}
\begin{proof}
  Let $\alpha$ be an arc of intersection of $D\cap E$ and suppose that
  $\alpha$ either has both endpoints in the same component of
  $\boundary E\setminus P$ or both endpoints in the same component of
  $\boundary E\cap P$.  We may assume $\alpha$ is an outermost such
  arc in $E$ which cuts off a disk $E'$ whose interior does not meet
  $D$.  Then we may surger $D$ along $E'$ to obtain a new
  $\boundary$--compressing disk meeting $E$ fewer times.


  Suppose that $\alpha$ has each endpoint in a different component of
  $\boundary E\cap P$.  Among all arcs in $D$ with endpoints on
  $\boundary D\cap P$ choose an outermost one $\alpha'$ so that
  $\alpha'$ cuts off a subdisk $D'$ of $D$ such that the interior of
  $D'$ does not meet $E$.  We may assume that $\alpha'$ does not have
  both endpoints in the same component of $\boundary E\cap P$ in $E$
  since if it did, we could find a simpler $\boundary$--compressing
  disk as in the previous paragraph.  Surgering $E$ along $D'$ we
  obtain a nonseparating disk which must be isotopic to $E$
  by~\autoref*{E-is-unique}.  However, this disk
  contradicts~\autoref*{minimum-of-E-cap-P}.

  Suppose then that $\alpha$ has one endpoint in a component of
  $\boundary E\cap P$ and the other in an adjacent component of
  $\boundary E\setminus P$.  We may assume that $\alpha$ is outermost
  in $E$ and cuts off a subdisk $E'$ of $E$ whose interior does not
  meet $D$.  By~\autoref*{boundary-D-cap-P}, $\boundary D$ separates
  $P$.  We may surger $D$ along $E'$ to obtain a new disk $D'$
  which meets $P$ in a single arc.  Since $\boundary D$ separates $P$
  and no arc of $\boundary E\cap P$ is separating, $\boundary D'\cap
  P$ is essential in $P$.  Therefore $D'$ is a new
  $\boundary$--compressing disk which meets $E$ fewer times.

  Finally, suppose that $\alpha$ has both endpoints on $\boundary
  E\setminus P$ and cuts off a disk $E'$ containing exactly one arc of
  $\boundary E\cap P$.  We may assume that $\alpha$ is outermost in
  $E$ so that it cuts off a subdisk $E'$ whose interior is disjoint
  from $D$.  As above, we can surger $D$ along $E'$ to obtain a new
  $\boundary$--compressing disk which has fewer intersections with $E$.
\end{proof}

Recall that the arc $\epsilon$ is defined as the largest subarc of
$\boundary D$ lying in $\boundary H$ and meeting $P$.  Let $\Delta =
\boundary D\setminus (\epsilon\cup\boundary N(K\cup t))$.  

\begin{lemma}\label{no-arcs-connect-epsilon-to-itself}
  There are no arcs of $D\cap E$ which have both endpoints in
  $\epsilon$.  There are no arcs of $D\cap E$ which have both
  endpoints in the same component of $\Delta$.
\end{lemma}
\begin{proof}
  There are no arcs with both endpoints in $\epsilon\cap P$
  by~\autoref*{arcs-in-E}.  Let $\alpha$ be a component of $D\cap E$
  which has both endpoints in $\epsilon$.  We may assume that $\alpha$
  is outermost in $D$ and cuts off a subdisk $D'$ with $\boundary
  D'=\alpha\cup\beta$, $\beta\subseteq\epsilon$.  

  If $\alpha$ has one endpoint in $\epsilon\cap P$ and the other in
  $\epsilon\cap R$, we may surger $E$ along $D'$ to obtain a
  nonseparating disk which must be isotopic to $E$
  by~\autoref*{E-is-unique}.  By~\autoref*{arcs-in-E} this disk meets
  $P$ fewer than $p+q-1$ times, and so it
  contradicts~\autoref*{minimum-of-E-cap-P}.  The same conclusion
  holds if $\alpha$ has both endpoints in $\epsilon\cap R$.

  If $\alpha$ has both endpoints in a component $\delta$ of $\Delta$,
  we may assume that $\alpha$ is outermost in $D$ and cuts off a
  subdisk $D'$ with $\boundary D'=\alpha\cup\beta$,
  $\beta\subseteq\delta$.  Surgering $E$ along $D'$ gives a disk
  violating~\autoref*{minimum-of-E-cap-P} as above.
\end{proof}

Note that~\autoref*{no-arcs-connect-epsilon-to-itself} shows that
$\Delta$ is nonempty.  Call two distinct components $\delta_1$ and
$\delta_2$ of $\Delta\cup\set{\epsilon}$ \defn{adjacent} if there is a
component $\phi$ of $\boundary D\cap \boundary N(K\cup t)$ such that
$\phi$ shares one endpoint with $\delta_1$ and the other with
$\delta_2$.  Orient $\boundary D$ so that each component of $\Delta$
inherits an orientation.  The \defn{label sequence} of a component
$\delta$ of $\Delta$ is the ordered sequence of arcs $\boundary E\cap
R$ that $\delta$ meets.  The \defn{reverse} of a label sequence $x_1,
x_2, \dots, x_n$ is $x_n, x_{n-1},\dots, x_1$.  

Let $\delta$ be a component of $\Delta$.  Then $\delta$ extends to a
\scc $\delta'$ in $R$ which is parallel to a component $\gamma$ of
$\boundary P$ by minimality of $(|D\cap E|, |\boundary D\cap\boundary
N(t)|)$.  By abuse of notation, we say that $\delta$ is parallel to
$\gamma$.

\begin{lemma}\label{label-sequence}
  All components of $\Delta$ are parallel on $R$, and no component of
  $\Delta$ is parallel to $\boundary_3P$.  
\end{lemma}
\begin{proof}
  Suppose that there are two components $\delta_1$ and $\delta_2$
  whose label sequences are different and not the reverse of one
  another.  By minimality of $|D\cap E|$, $\delta_1$ and $\delta_2$
  must be parallel to different components of $\boundary P$.
  Therefore $\boundary D$ meets every component of $\boundary E\cap
  R$.  But an arc of $D\cap E$, outermost in $E$, must then
  violate~\autoref*{arcs-in-E}.

  If there is a component of $\Delta$ which is parallel to
  $\boundary_3P$, then every arc of $\boundary E\cap R$ except for $c$
  contains the endpoint of an arc of $D\cap E$.  Since there are at
  least two outermost arcs of $D\cap E$ in $E$, there is at least one
  which violates~\autoref*{arcs-in-E}.
\end{proof}

Let $D'$ be a $2n$--gon, $n>1$, in $D$ whose boundary consists of a
union of components of $\boundary D\setminus\boundary E$ and
components of $D\cap E$ and whose interior is disjoint from $E$.  At
each vertex of $D'$ we see a label $a^P$, $b_i^W$, $c^R$, or $d_l^W$,
where $W=R$ or $P$.  Suppose that the vertex set of $D'$ contains
exactly two labels $x$ and $y$, and suppose further that as we go
around $\boundary D'$ these labels alternate.  We call $\boundary D'$
a \SC.  An example is shown in~\autoref*{fig:SC}. The bold arcs
represent $\boundary D\cap\boundary N(K\cup t)$ and the shaded region
is $D'$.  Note that since the vertex set of $D'$ contains exactly two
labels and there are no arcs of $D\cap E$ with both endpoints in
$\boundary D\cap P$, we have $\boundary D'\cap P=\emptyset$.

\begin{figure}[h!tb]
  \begin{center}
    \def\svgwidth{0.3\textwidth}
    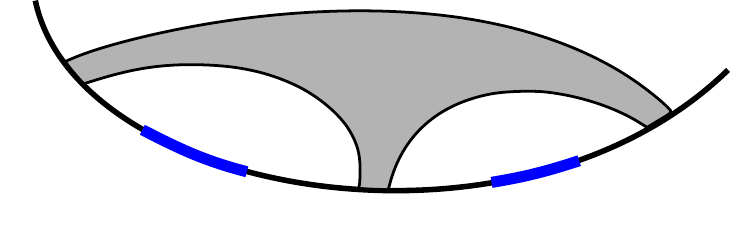
  \end{center}
  \caption{A \SC appearing in $D$}
  \label{fig:SC}
\end{figure}

Note that the arcs of $\boundary D'\cap E$ are parallel in $E$, and
the arcs of $\boundary D'\cap\boundary D$ are parallel in $R$.  Let
$A_E$ be a subdisk of $E$ whose boundary consists of two subarcs of
$\boundary E$ and two components of $\boundary D'\cap E$ such that
$A_E$ contains all components of $\boundary D\cap E$.  Similarly, let
$A_R$ be a subdisk of $R$ whose boundary consists of two subarcs of
$\boundary E$ and two components of $\boundary D'\cap \boundary D$
such that $A_R$ contains all components of $\boundary D'\cap\boundary
D$.  Then $S=N(A_E\cup A_R)$ is a solid torus in $H'$.  Since $n>1$,
$N(S\cup D')$ is a punctured lens space.  (We can see this by noting
that $N(S\cup D')$ is obtained by attaching the $2$--handle $N(D')$ to
a solid torus so that the attaching curve meets a meridian
algebraically more than once.)

In particular, $H'$ is reducible, which is impossible since we are
assuming that $H'$ is a handlebody.  Therefore such disks $D'$ do not
exist in $D$.  We can use this criterion to show:

\begin{lemma}\label{adjacent-cpts-of-delta}
 Let $\delta_1$ and $\delta_2$ be adjacent components of $\Delta$,
  and suppose that $\delta_1$ has label sequence $x_1,\dots,x_n$.
  Then there are at most $\floor{n/2}$ arcs of $D\cap E$ with one
  endpoint in $\delta_1$ and the other in $\delta_2$.
\end{lemma}
\begin{proof}
  Suppose not.  Then there are two adjacent components $\delta_1$ and
  $\delta_2$ of $\Delta$ with more than $\floor{n/2}$ arcs of $D\cap
  E$ connecting them.  By~\autoref*{no-arcs-connect-epsilon-to-itself}
  there are no arcs with both endpoints in $\delta_i$, $i=1,2$, and
  therefore these arcs are all parallel in $D$.
  By~\autoref*{label-sequence}, $\delta_2$ has the label sequence of
  $\delta_1$ or its reverse.

  If $\delta_2$ has the reverse label sequence of $\delta_1$, then
  there is an arc both of whose endpoints have the same label.  This
  contradicts~\autoref*{arcs-in-E}.

  So suppose that $\delta_2$ has the same label sequence as
  $\delta_1$.  If $n$ is odd, there is an arc both of whose endpoints
  have label $x_{(n+1)/2}$, and this arc
  violates~\autoref*{arcs-in-E}.  On the other hand, if $n$ is even,
  then there is a Scharlemann cycle with label set $\set{x_{n/2},
    x_{n/2+1}}$.  
\end{proof}

\begin{lemma}\label{SC}
  There are no arcs of $D\cap E$ with endpoints in
  nonadjacent components of $\Delta$.
\end{lemma}
\begin{proof}
  By~\autoref*{label-sequence}, every component of $\Delta$ has the
  label sequence of $\delta$ or its reverse.  Suppose there were an
  arc of $D\cap E$ with endpoints in nonadjacent components of
  $\Delta$, and choose $\alpha$ to be one outermost in $D$.  The arc
  $\alpha$ cuts off a subdisk $D'$ in which every arc $\interior{D'}\cap E$
  connects adjacent components of $\Delta' = \Delta\cap \boundary D'$.

  If two components of $\Delta'$ have label sequences which are the
  reverse of one another, then we can find two such components
  $\delta_1$ and $\delta_2$ which are adjacent.  Suppose there is an
  arc of $\interior{D'}\cap E$ connecting them.  Then an outermost
  such arc in $D'$ violates~\autoref*{arcs-in-E}.  If no arcs of
  $\interior{D'}\cap E$ connect $\delta_1$ and $\delta_2$, then there
  is a component $\delta_3$ of $\Delta'$ adjacent to $\delta_1$ such
  that every arc with one endpoint in $\delta_1$ has its other
  endpoint in $\delta_3$.  This
  contradicts~\autoref*{adjacent-cpts-of-delta}.

  Therefore all components of $\Delta'$ which are also components of
  $\Delta$ have the same label sequence; let $n$ be the length of this
  sequence.  By~\autoref*{adjacent-cpts-of-delta}, $n$ must be even
  and every component of $\Delta'$ must contain exactly $n/2$ parallel
  arcs connecting an adjacent component.  But in this case, the
  $2|\Delta'|$--gon which appears in $D'$ is a \SC.  This is
  impossible by the remarks following~\autoref*{label-sequence}.
\end{proof}

\begin{lemma}\label{minimum-arcs-in-epsilon}
  Let $\delta$ be a component of $\Delta$ with label sequence
  $x_1,x_2,\dots,x_n$.  Then there are at least $2\ceil{n/2}$
  endpoints of arcs of $D\cap E$ in $\epsilon$.
\end{lemma}
\begin{proof}
  No arcs of $D\cap E$ connect nonadjacent components of $\Delta$
  by~\autoref*{SC}.  Since at most $\floor{n/2}$ arcs connect adjacent
  components by~\autoref*{adjacent-cpts-of-delta}, there must be at
  least $2\ceil{n/2}$ arcs which have one endpoint in $\epsilon$.
\end{proof}

\begin{lemma}\label{epsilon-meets-E}
  The arc $\epsilon$ meets $\boundary E$ in $R$.
\end{lemma}
\begin{proof}
  Suppose first that $\boundary D$ is nonseparating in $P$ so that
  $\boundary D\cap \boundary E\cap P=\emptyset$.  Then $\boundary
  E\cap R\neq\emptyset$ by~\autoref*{minimum-arcs-in-epsilon}.

  If $\boundary D$ separates $P$, then by minimality of $|\boundary
  D\cap \boundary E\cap P|$, we know exactly in which regions of $R$
  the points $\epsilon\cap\boundary P$ lie.  The three cases are
  pictured in~\autoref*{fig:E-cap-D-separating}, and we see that
  the endpoints of $\epsilon\cap P$ lie in different regions of $R$.
  Therefore $\epsilon$ must cross $\boundary E$ in $R$.
\end{proof}


Recall that we are assuming $1<p<q$.

\begin{lemma}\label{endpoints-on-gamma}
  Let $\delta$ be a component of $\Delta$. Then there are at least
  $p$ arcs of intersection $D\cap E$ with endpoints in
  $\delta$.
\end{lemma}
\begin{proof}
  By the previous remarks, $\delta$ is parallel to a component of
  $\boundary P$, and these meet $E$ minimally $p$ times.
\end{proof}

\begin{lemma}\label{outermost-in-D}
  There is an arc $\alpha$ of $D\cap E$, outermost in $D$, with one
  endpoint in $\epsilon$ and the other in an adjacent component of
  $\Delta$.  This arc cuts off a disk $D'$ of $D$ with $\boundary
  D'=\alpha\cup\beta$.  The endpoints of $\beta$ lie in $d_{\inv{q}}^R$ and
  $b_{\inv{p}}^R$ or $d_{(p-1)\inv{q}}^R$ and $b_{(q-1)\inv{p}}^R$
  depending on whether the foot of the tunnel lies in $H_1$ or $H_2$.
  Furthermore, $\beta$ does not meet $P$.
\end{lemma}
\begin{proof}
  The first part follows from~\autoref*{no-arcs-connect-epsilon-to-itself}
  and~\autoref*{epsilon-meets-E}.  If $\alpha$ did not have an
  endpoint in a component of $\Delta$ adjacent to $\epsilon$ it could
  not be outermost by~\autoref*{endpoints-on-gamma}.

  If $\beta\cap P\neq\emptyset$, there is at least one other arc of
  $D\cap E$ with an endpoint in $\epsilon\cap R$
  by~\autoref*{minimum-arcs-in-epsilon}.  An outermost such arc gives
  the desired $\alpha$ by~\autoref*{adjacent-cpts-of-delta}
  and~\autoref*{SC}.
  
  If $\beta$ has endpoints on the same component of $\boundary E\cap
  R$ in $R$, we obtain an arc in $E$
  contradicting~\autoref*{arcs-in-E}.  Furthermore, since $\alpha$ is
  outermost in $D$, we may surger $E$ along a subdisk of $D$ to obtain
  a new disk which does not meet $E$.  Therefore $\beta$ has opposite
  signs of intersection with $\boundary E$ in $R$.

  Orienting $\boundary E$ we see that all arcs of $\boundary E\cap R$
  in $R$ which connect two boundary components are coherently
  oriented.  Therefore $\beta$ must lie in one of the hexagonal
  regions.  Checking the possibilities, we see that the condition that
  $\beta$ has opposite signs of intersection with $\boundary E$
  implies that it meets either $d_{\inv{q}}^R$ and
  $b_{\inv{p}}^R$ or $d_{(p-1)\inv{q}}^R$ and $b_{(q-1)\inv{p}}^R$.
\end{proof}




\section{\texorpdfstring{Families of knots which are not
    $1$--bridge}{Families of knots which are not 1--bridge}}
Fix a nontrivial $2$--bridge knot $L$.  Let $q=p+1$ and write $(M, H,
K) = (M_{p,p+1}^L, H_{p,p+1}^L, K_{p,p+1}^L)$.  Suppose that there is
a tunnel $t$ as in~\autoref*{1-bridge-iff-1-tunnel} so that
$M'=\closure{M\setminus (K\cup t)}$ is a handlebody, and suppose that
$t$ lies entirely in $H$.  

\begin{lemma}\label{easy-case}
  Any $\boundary$--compressing disk for $P$ in $M'$
  violates~\autoref*{three-punctured-spheres}.  
\end{lemma}
\begin{proof}
  We only need to consider $\boundary$--compressing disks for $P$ in
  $H'=\closure{H\setminus N(K\cup t)}$ which meet $E$
  by~\autoref*{core-not-prim}.  So let $D$ be a
  $\boundary$--compressing disk for $P$ in $H'$ which minimizes
  $(|D\cap E|, |D\cap\boundary N(t)|, |\boundary D\cap \boundary E\cap
  P|)$, and let $D'$ be as in~\autoref*{outermost-in-D}.  Examining
  the options given by that lemma, we see that $\alpha$ must either
  have one endpoint in $b_{-1}$ and the other in $d_1$ or one endpoint
  in $b_1$ and the other in $d_{-1}$.  An arc with endpoints in $b_1$
  and $d_{-1}$ violates~\autoref*{arcs-in-E}, so $\alpha$ must be of
  the first type.

  The arc $\alpha$ cuts off a subdisk $E'$ of $E$ which meets $P$ in
  exactly two arcs ($b_{-1}^P$ and $d_1^P$).  An arc in $E'$ with one
  endpoint in $c$ would violate~\autoref*{arcs-in-E}, and so no
  component $\delta$ of $\Delta$ meets $c$.  Therefore every component
  of $\Delta$ is parallel to $\boundary_3P$, but this
  contradicts~\autoref*{label-sequence}.
\end{proof}

Now we tackle the more difficult case when $q=2p\pm 1$, $p>1$, and
$q>3$.  The intersection patterns of $\boundary E\cap P$ and
$\boundary E\cap R$ are special in these cases, and they are pictured
in~\autoref*{fig:PR-special-case}.  The disk $E$ remains as
in~\autoref*{fig:E}.  There are two hexagonal regions in each of $P$
and $R$; we will refer to these regions as $H_1^P$, $H_2^P$, $H_1^R$,
and $H_2^R$.  Note that when $q=2p-1$, $H_1$ and $H_2$ are reversed
from~\autoref*{fig:PR}.

\begin{figure}[h!tb]
  \begin{center}
    \def\svgwidth{0.9\textwidth}
    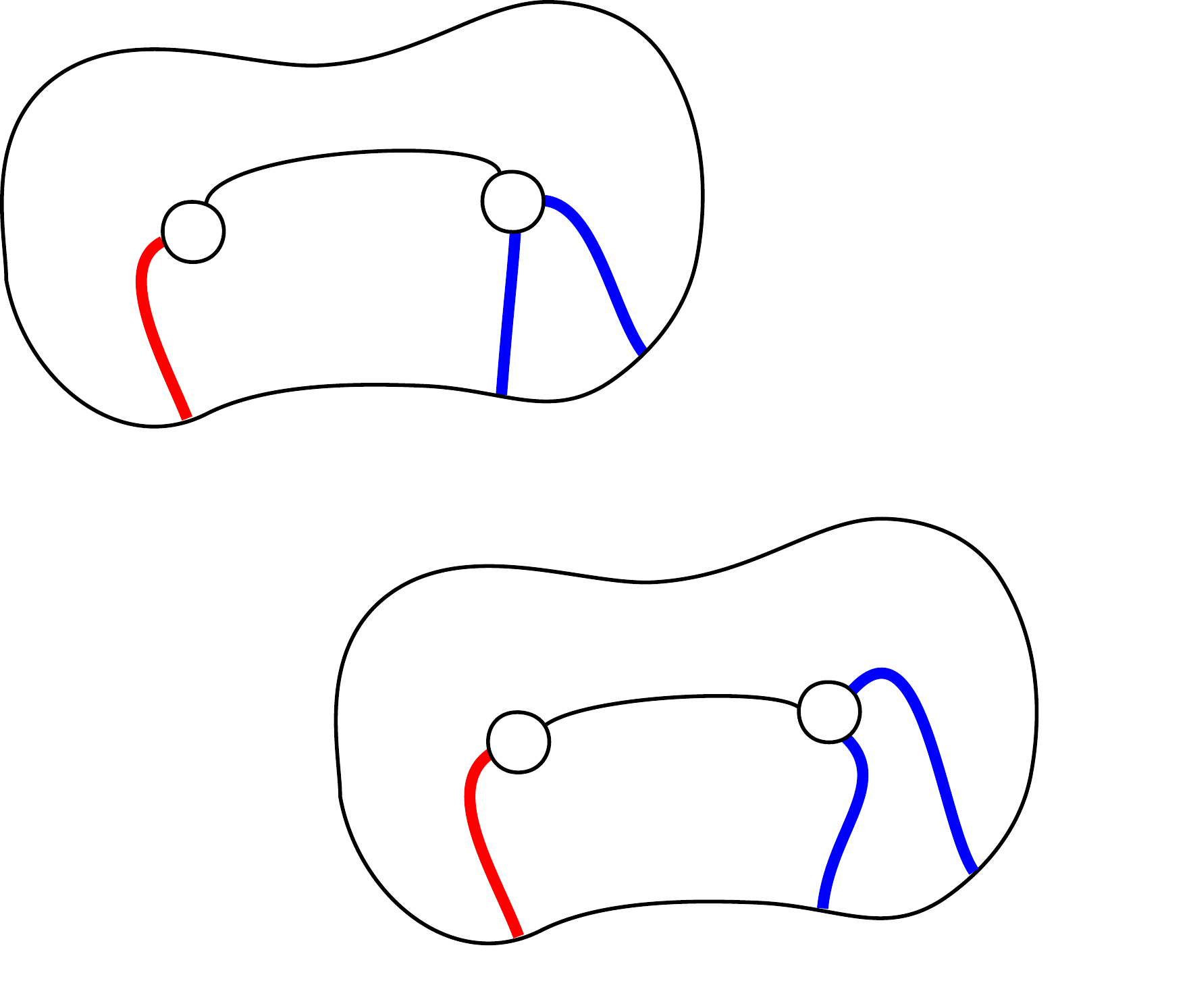
  \end{center}
  \caption{The surfaces $P$ and $R$ in the case $q=2p\pm 1$}
  \label{fig:PR-special-case}
\end{figure}


As before, let $D$ be a $\boundary$--compressing disk for $P$ in $H'$
which minimizes $(|D\cap E|, |D\cap\boundary N(t)|, |\boundary D\cap
\boundary E\cap P|)$.  Let $\beta$ be as in~\autoref*{outermost-in-D}.
Examining the possibilities given by that lemma, we see that $\beta$
must have endpoints on $d_{-1}$ and $b_2$, or $d_1$ and $b_{-2}$.  The
two possibilities for $\alpha\subseteq E$ are shown
in~\autoref*{fig:alpha-E}.



\begin{figure}[h!tb]
  \begin{center}
    \def\svgwidth{0.5\textwidth}
    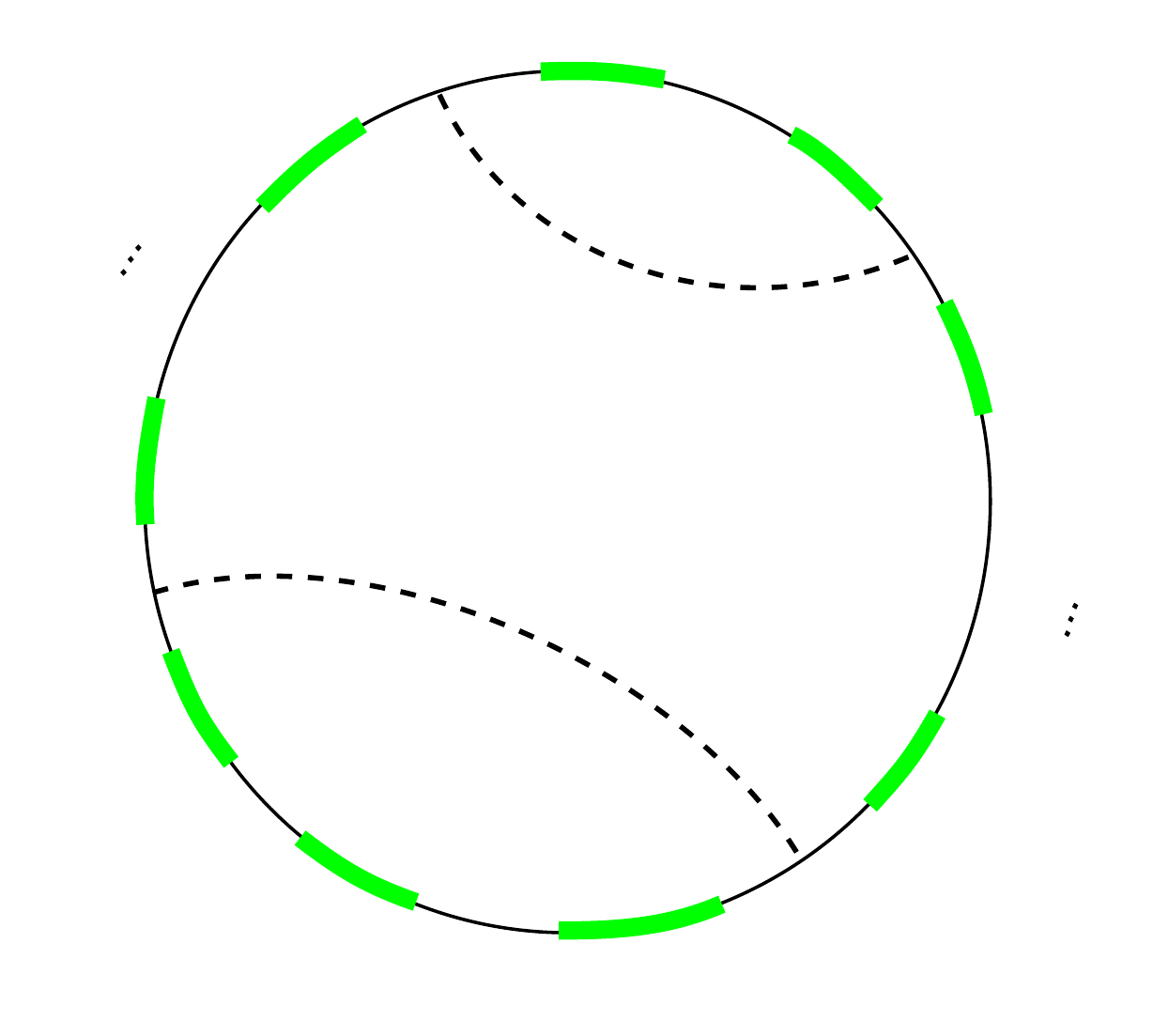
  \end{center}
  \caption{Possibilities for $\alpha$ in $E$}
  \label{fig:alpha-E}
\end{figure}

Let $R'=R\cap H'$.

\begin{lemma}\label{endpoints-of-epsilon}
  When $q=2p\pm 1$, the arc $\epsilon$ has endpoints in different
  components of $R'\setminus\Delta$.
\end{lemma}
\begin{proof}
  Suppose that $\epsilon$ has its endpoints in the same component of
  $R'\setminus\Delta$.  Then $D$ extends to a properly embedded
  surface $S$ in $\closure{H\setminus N(K)}$ by attaching bands along
  $\boundary D$.  This surface has $|\Delta|+1$ boundary components on
  $\boundary H$, $|\Delta|$ of which are parallel to a component of
  $\boundary P$ and one of which meets $P$ in a single essential arc.
  Note that since the bands are attached in an orientation preserving
  way, $S$ is orientable.  Furthermore, we can see that $S$ is a
  punctured sphere by calculating $\chi(S)$.

  The surface $S$ has $|\Delta|$ boundary components parallel to a
  component $\boundary P$.  Let $\gamma$ be a \scc in $R$ which does
  not meet $\epsilon$ and is parallel to this component.  Note that
  $\gamma$ cannot be parallel to $\boundary_3P$
  by~\autoref*{label-sequence}.  After attaching a $2$--handle to $H$
  along $\gamma$, $S$ extends to an annulus $A$ in a solid torus $U$.
  Let $Q$ be the annulus in $\boundary U$ consisting of $P$ together
  with the trivial disk bounded by the component of $\boundary P$
  which was made trivial by the $2$--handle attachment.  Note that $Q$
  is longitudinal on $\boundary U$.  One component of $\boundary A$
  lies on $\boundary U$.  Call it $\boundary_1A$.  The other,
  $\boundary_2A$, lies on $\boundary N(K)$.

  We want to show that either $\boundary_1A$ is parallel to a
  component of $\boundary Q$ or meets $Q$ in an essential arc.  If
  $\boundary D$ is nonseparating in $P$, then $\boundary_1A$ meets $Q$
  in an essential arc, and so it is certainly nontrivial in $\boundary
  U$. So suppose that $\boundary D$ separates $P$.  The possibilities
  for $\boundary D\cap P$ appear in~\autoref*{fig:E-cap-D-separating}.
  If $\epsilon$ meets the component of $\boundary P$ parallel to
  $\gamma$, it is clear that we can isotop $\boundary_1A$ to be
  parallel to a component of $\boundary Q$.  Using this fact, there
  are only a few remaining cases to check:

  \begin{enumerate}
  \item If $\boundary D$ meets $\boundary_3P$ and $\gamma$ is parallel
    to $\boundary_1P$, then $\epsilon$ meets either $b_2^R$ or
    $b_{-2}^R$ by~\autoref*{outermost-in-D}, depending on whether the
    tunnel exits in $H_2^R$ or $H_1^R$.  In either case, the
    minimality of $|D\cap E|$ implies that $\boundary D$ meets every
    component of $\boundary E\cap R$.  An outermost arc of $D\cap E$
    in $E$ must then violate~\autoref*{arcs-in-E}.  If $\boundary D$ meets
    $\boundary_3P$ and $\gamma$ is parallel to $\boundary_2P$, then
    $\epsilon$ meets either $d_1^R$ or $d_{-1}^R$ and the same
    argument applies.
  \item If $\boundary D$ meets $\boundary_2 P$, we see
    from~\autoref*{fig:E-cap-D-separating} that $\epsilon$ meets the
    components of $\boundary_2P\setminus\boundary E$ which lie on
    either side of $b_1^R$.  If $\gamma$ is parallel to $\boundary_1P$
    and the tunnel exits in $H_2^R$, then $\epsilon$ meets
    $b_{2}^R,b_4^R,\dots, b_{-1}^R$.  If $\epsilon$ meets $b_1^R$, it
    is easy to see from~\autoref*{fig:alpha-E} that there is an arc
    forbidden by~\autoref*{arcs-in-E}.  However, if $\epsilon$ does
    not meet $b_1^R$, then $\boundary D$ meets every arc of $\boundary
    E\cap R$ except for $b_1^R$, and again there must be an outermost
    arc of $D\cap E$ in $E$ which violates~\autoref*{arcs-in-E}.
  \item If $\boundary D$ meets $\boundary_2P$, $\gamma$ is parallel to
    $\boundary_1P$, and the tunnel exits in $H_1^R$, then $\epsilon$
    meets $b_{-2}^R,\dots,b_3^R$ by~\autoref*{outermost-in-D}.  Note
    that $\alpha\subseteq E$ has endpoints in $b_{-2}^R$ and $d_1^R$.
    Furthermore, $\boundary D$ meets $c^R$.
    Examining~\autoref*{fig:alpha-E} we see that $\epsilon$ cannot
    meet $b_{-1}^R$ by~\autoref*{arcs-in-E}.  Therefore $\epsilon$
    meets $b_{-2}^R$ exactly twice.

    Examining~\autoref*{fig:alpha-E} we see that~\autoref*{arcs-in-E}
    implies that any arc of $D\cap E$ with one endpoint in $c^R$ must
    have its other endpoint in either $b_{-2}^R$ or $b_{-2}^P$.  The
    latter is not a possibility since $\boundary D\cap P$ does not
    meet $b_{-2}^P$ in this case.  In $D$, an arc with endpoint $c^R$
    in a component of $\Delta$ has its other endpoint in $b_{-2}^R$ in
    $\epsilon$.  Since $\epsilon$ meets $b_{-2}^R$ exactly twice and
    the other arc has endpoints $b_{-2}^R$ and $d_1^R$, there is
    exactly one component of $\Delta$.  This component contains $p$
    endpoints of arcs of intersection of $D\cap E$.  However,
    $\epsilon$ contains at least $q+p-3$ endpoints of arcs of
    intersection of $D\cap E$ ($\epsilon\cap P$ contains $p-1$
    endpoints of arcs, and $\epsilon\cap R$ contains at least $q-2$).
    Since $q>3$, there is an arc
    violating~\autoref*{no-arcs-connect-epsilon-to-itself}.

  \item If $\boundary D$ meets $\boundary_1P$, $\gamma$ is parallel to
    $\boundary_2P$, and the tunnel exits in $H_2^R$, then $\gamma$
    meets $b_1^R$.  Recall that in this case $\beta$ has endpoints in
    $d_{-1}^R$ and $b_2^R$.  Examining~\autoref*{fig:alpha-E} it is clear that
    there is an arc violating~\autoref*{arcs-in-E}.
  \item If $\boundary D$ meets $\boundary_1P$, $\gamma$ is parallel to
    $\boundary_2P$, and the tunnel exits in $H_1^R$, then $\gamma$
    meets both $c^R$ and $b_{-1}^R$.  In this case $\beta$ has
    endpoints in $d_1^R$ and $b_{-1}^R$, and examining~\autoref*{fig:alpha-E}
    there must be an arc violating~\autoref*{arcs-in-E}.
  \end{enumerate}

  Therefore $\boundary_1A$ is nontrivial in $\boundary U$.  Note that
  $\closure{U\setminus N(K)}$ is irreducible and not a product
  by~\autoref*{attaching-2-h}.  By~\cite[Lemma 2.5.3]{CGLS},
  $\boundary_2A$ is either meridional on $\boundary N(K)$ or meets a
  meridian of $\boundary N(K)$ exactly once.  Suppose that
  $\boundary_2A$ is not meridional on $\boundary N(K)$.  By the above
  arguments, $A$ describes an isotopy of $K$ to $\boundary U$ so that
  $K$ meets a component of $\boundary Q$ zero or one times.

  We want to examine this annulus in the context of the original
  construction of $K\subseteq H$ from~\autoref*{chapter:knots}.
  There, it is clear that after attaching a $2$--handle to
  $\boundary_1P$ or $\boundary_2P$ we obtain a knot which is isotopic
  to a curve in the boundary of the solid torus meeting the other
  components of $\boundary P$ either $p$ or $q$ times.  Thus there is
  an annulus $A'$ properly embedded in $\closure{U\setminus N(K)}$
  which meets a component of $\boundary Q$ either $p$ or $q$ times.
  Since there are two annuli in the knot exterior with different
  slopes on $\boundary N(K)$, the proof of~\cite[Lemma
  2.5.3]{CGLS} shows that $\closure{U\setminus N(K)}$ must be a
  product.  However, we have already seen that this is not the case.
  


  Therefore $\boundary_2A$ is meridional on $\boundary N(K)$.  We
  obtain a meridian disk for $U$ meeting $K$ once, which is impossible
  by~\autoref*{attaching-2-h} and the fact that $\gamma$ is not
  parallel to $\boundary_3P$.
\end{proof}

Now let $q=2p\pm 1$, $p>1$, and $q>3$, and write $(M, H, K) =
(M_{p,q}^L, H_{p,q}^L, K_{p,q}^L)$.  As before, suppose that there is
a tunnel $t$ as in~\autoref*{1-bridge-iff-1-tunnel} so that
$M'=\closure{M\setminus (K\cup t)}$ is a handlebody, and suppose that
$t$ lies entirely in $H$.

\begin{lemma}\label{hard-case}
  The surface $P\subseteq M'$ is $\boundary$--incompressible.
\end{lemma}
\begin{proof}
  By~\autoref*{core-not-prim}, we need only examine
  $\boundary$--compressing disks which meet $E$.  Let $D$ be such a
  $\boundary$--compressing disk for $P$ in $H'$ which minimizes
  $(|D\cap E|, |D\cap\boundary N(t)|, |\boundary D\cap \boundary E\cap
  P|)$.  By~\autoref*{outermost-in-D} there is a component $\alpha$ of
  $D\cap E$, outermost in $D$, which meets either $d_{-1}$ and $b_2$
  or $d_1$ and $b_{-2}$ in $R$.

  \textbf{Case 1:} Suppose that $D$ is a $\boundary$--compressing disk
  as above, $\boundary D$ is nonseparating in $P$, and that $\alpha$
  has endpoints in $d_{-1}$ and $b_2$.  Recall that the proof
  of~\autoref*{outermost-in-D} shows that both components
  $\epsilon\cap R$ and $\epsilon\cap P$ are contained in the hexagonal
  regions of $R$ and $P$, respectively.  Let $\delta$ be the component
  of $\Delta$ meeting $\beta$.  Examining~\autoref*{fig:alpha-E} we
  see that $\boundary D$ cannot meet $b_1^R$ without
  violating~\autoref*{arcs-in-E}.  Therefore $\delta$ is parallel to
  $\boundary_1P$, and by~\autoref*{label-sequence} every component of
  $\Delta$ has label sequence $c,d_1,\dots,d_{-1}$ or the reverse.
  Note that $\boundary_2P$ and $\boundary_3P$ lie in the same
  component of $R'\setminus\Delta$.  The arc $\epsilon$ meets $b_2$,
  and so by~\autoref*{endpoints-of-epsilon} and the fact that
  $\boundary D$ is nonseparating, $\epsilon$ must meet $\boundary_1P$.
  Recall from~\autoref*{P-R-remarks} that $\boundary_2P\cap
  a^P=\boundary_2P\cap b_1^R$.  A look
  at~\autoref*{fig:PR-special-case} shows that on the side of $\beta$,
  $\epsilon$ must cross $b_2,b_4,\dots b_{-1}$.

  If $\epsilon$ does not meet any other arcs of $\boundary E\cap R$,
  then by~\autoref*{no-arcs-connect-epsilon-to-itself},
  \autoref*{adjacent-cpts-of-delta}, and~\autoref*{SC}, there is an
  outermost arc of $D$ connecting $b_{-1}$ to either $c$ or $d_{-1}$.
  The first arc is forbidden by~\autoref*{arcs-in-E}.  Therefore there
  is an arc in $E$ with endpoints in $b_{-1}$ and $d_{-1}$.  Note that
  all arcs in $E$ connecting $b_2,b_4,\dots b_{-3}$ to $c$ or any
  $d_i$ must cross this arc unless all the $b_{2i}$ connect to
  $d_{-1}$.  Since $q>4$, we may consider the subdisk of $D$ bounded
  by $\boundary D$, $\alpha$, and the arc with endpoints $b_4$ and
  $d_{-1}$.  If $p=2$, then $d_1^R=d_{-1}^R$.  In light
  of~\autoref*{SC}, there must be an arc in this subdisk
  violating~\autoref*{arcs-in-E}.  If $p>2$, then by~\autoref*{SC} the
  subdisk contains arcs violating~\autoref*{adjacent-cpts-of-delta}.

  Therefore $\epsilon$ must meet other arcs of $\boundary E\cap R$.
  By~\autoref*{P-R-remarks}, it wraps around $\boundary_1P$ at least
  once.  The component of $\epsilon\setminus P$ meeting $d_{-1}$ has
  the same label set as the components of $\Delta$.  We may treat this
  subarc of $\epsilon$ as another component of $\Delta$, and the
  arguments of~\autoref*{no-arcs-connect-epsilon-to-itself},
  \autoref*{adjacent-cpts-of-delta}, and~\autoref*{SC} give a
  contradiction.


  \textbf{Case 2:} Suppose that $\boundary D$ is nonseparating in $P$
  and $\alpha$ has endpoints in $d_1$ and $b_{-2}$.  Let $\delta$ be
  the component of $\Delta$ meeting
  $\beta$. By~\autoref*{label-sequence}, $\delta$ must meet $c^R$.
  Examining~\autoref*{fig:alpha-E} we see that it misses $b_{-1}$
  by~\autoref*{arcs-in-E}.  Therefore $\delta$ is parallel to
  $\boundary_1P$, and by~\autoref*{label-sequence} the label sequence
  of every component of $\Delta$ is $d_1,\dots, d_{-1}, c$ or the
  reverse.  Again, $\boundary_2P$ and $\boundary_3P$ lie in the same
  component of $R'\setminus\Delta$.  We see
  from~\autoref*{endpoints-of-epsilon} that $\epsilon$ must either
  meet $b_{-2},b_{-4},\dots, b_3$ before passing to $P$, or it must meet
  $b_{-2},b_{-4},\dots, b_1$ before passing to $P$.

  Suppose that $\epsilon\cap R$ consists of exactly $b_{-2},\dots,
  b_3$.  Then $d_{-1}^R=d_1^R$, and so $(p, q)=(2, 5)$.  In this case,
  $b_3$ is the only endpoint of a component of $D\cap E$ in
  $\epsilon$, and this violates~\autoref*{minimum-arcs-in-epsilon}.  

  If $p>2$, then $\epsilon$ ends after meeting either $c$ or $d_1$.
  The arc with this endpoint must be outermost
  by~\autoref*{adjacent-cpts-of-delta} and~\autoref*{SC}.  Therefore
  there is an arc in $E$ with both endpoints on the same component of $\boundary
  E\cap R$ or one with endpoints $c$ and $d_1$, both of which
  contradict~\autoref*{arcs-in-E}.



  So suppose that $\epsilon$ hits $b_{-2},\dots, b_1$ before traveling
  to $P$.  As above, the last endpoint of $\epsilon$ must be either
  $c$ or $d_1$, and we get a contradiction with~\autoref*{arcs-in-E}.

  \textbf{Case 3:} If $\boundary D$ separates $P$, note that the
  points $\epsilon\cap\boundary R$ lie on the same component of
  $\boundary R$.  But this means that both components of $\epsilon\cap
  R'$ lie in the same component of $R'\setminus\Delta$, which
  violates~\autoref*{endpoints-of-epsilon}.
\end{proof}

\bibliographystyle{plain}
\bibliography{thesis}%

\end{document}

%% file: figures/H_3_4-with-P.pdf_tex

\begingroup
  \makeatletter
  \providecommand\color[2][]{%
    \errmessage{(Inkscape) Color is used for the text in Inkscape, but the package 'color.sty' is not loaded}
    \renewcommand\color[2][]{}%
  }
  \providecommand\transparent[1]{%
    \errmessage{(Inkscape) Transparency is used (non-zero) for the text in Inkscape, but the package 'transparent.sty' is not loaded}
    \renewcommand\transparent[1]{}%
  }
  \providecommand\rotatebox[2]{#2}
  \ifx\svgwidth\undefined
    \setlength{\unitlength}{512.49779187pt}
  \else
    \setlength{\unitlength}{\svgwidth}
  \fi
  \global\let\svgwidth\undefined
  \makeatother
  \begin{picture}(1,0.68573775)%
    \put(0,0){\includegraphics[width=\unitlength]{H_3_4-with-P.pdf}}%
  \end{picture}%
\endgroup

%% file: figures/P-invariant-involution.pdf_tex
\begingroup%
  \makeatletter%
  \providecommand\color[2][]{%
    \errmessage{(Inkscape) Color is used for the text in Inkscape, but the package 'color.sty' is not loaded}%
    \renewcommand\color[2][]{}%
  }%
  \providecommand\transparent[1]{%
    \errmessage{(Inkscape) Transparency is used (non-zero) for the text in Inkscape, but the package 'transparent.sty' is not loaded}%
    \renewcommand\transparent[1]{}%
  }%
  \providecommand\rotatebox[2]{#2}%
  \ifx\svgwidth\undefined%
    \setlength{\unitlength}{426.58864746bp}%
    \ifx\svgscale\undefined%
      \relax%
    \else%
      \setlength{\unitlength}{\unitlength * \real{\svgscale}}%
    \fi%
  \else%
    \setlength{\unitlength}{\svgwidth}%
  \fi%
  \global\let\svgwidth\undefined%
  \global\let\svgscale\undefined%
  \makeatother%
  \begin{picture}(1,0.5354104)%
    \put(0,0){\includegraphics[width=\unitlength]{P-invariant-involution.pdf}}%
    \put(-0.0011263,0.19914559){\color[rgb]{0,0,0}\makebox(0,0)[lb]{\smash{$D_l$}}}%
    \put(0.94329102,0.19914559){\color[rgb]{0,0,0}\makebox(0,0)[lb]{\smash{$D_m$}}}%
    \put(0.63702998,0.35834734){\color[rgb]{0,0,0}\makebox(0,0)[lb]{\smash{$\boundary_3 P$}}}%
    \put(0.40497317,0.19914559){\color[rgb]{0,0,0}\makebox(0,0)[lb]{\smash{$\boundary_1 P$}}}%
    \put(0.51560484,0.25850894){\color[rgb]{0,0,0}\makebox(0,0)[lb]{\smash{$\boundary_2 P$}}}%
  \end{picture}%
\endgroup%

%% file: figures/E-2.pdf_tex
\begingroup%
  \makeatletter%
  \providecommand\color[2][]{%
    \errmessage{(Inkscape) Color is used for the text in Inkscape, but the package 'color.sty' is not loaded}%
    \renewcommand\color[2][]{}%
  }%
  \providecommand\transparent[1]{%
    \errmessage{(Inkscape) Transparency is used (non-zero) for the text in Inkscape, but the package 'transparent.sty' is not loaded}%
    \renewcommand\transparent[1]{}%
  }%
  \providecommand\rotatebox[2]{#2}%
  \ifx\svgwidth\undefined%
    \setlength{\unitlength}{353.61549072bp}%
    \ifx\svgscale\undefined%
      \relax%
    \else%
      \setlength{\unitlength}{\unitlength * \real{\svgscale}}%
    \fi%
  \else%
    \setlength{\unitlength}{\svgwidth}%
  \fi%
  \global\let\svgwidth\undefined%
  \global\let\svgscale\undefined%
  \makeatother%
  \begin{picture}(1,0.89819064)%
    \put(0,0){\includegraphics[width=\unitlength]{E-2.pdf}}%
    \put(0.50422562,0.87756442){\color[rgb]{0,0,0}\makebox(0,0)[lb]{\smash{$a^P$}}}%
    \put(0.63893328,0.86557437){\color[rgb]{0,0,0}\makebox(0,0)[lb]{\smash{$b_1^R$}}}%
    \put(0.75793944,0.80482257){\color[rgb]{0,0,0}\makebox(0,0)[lb]{\smash{$b_1^P$}}}%
    \put(0.84514395,0.70803521){\color[rgb]{0,0,0}\makebox(0,0)[lb]{\smash{$b_2^R$}}}%
    \put(0.89862574,0.58771595){\color[rgb]{0,0,0}\makebox(0,0)[lb]{\smash{$b_2^P$}}}%
    \put(0.91621151,0.46840249){\color[rgb]{0,0,0}\makebox(0,0)[lb]{\smash{$b_3^R$}}}%
    \put(0.83234425,0.195732){\color[rgb]{0,0,0}\makebox(0,0)[lb]{\smash{$b_{-3}^P$}}}%
    \put(0.72872864,0.09508778){\color[rgb]{0,0,0}\makebox(0,0)[lb]{\smash{$b_{-2}^R$}}}%
    \put(0.54917514,0.01350489){\color[rgb]{0,0,0}\makebox(0,0)[lb]{\smash{$b_{-2}^P$}}}%
    \put(0.38091349,0.0035791){\color[rgb]{0,0,0}\makebox(0,0)[lb]{\smash{$b_{-1}^R$}}}%
    \put(0.23490239,0.05202398){\color[rgb]{0,0,0}\makebox(0,0)[lb]{\smash{$b_{-1}^P$}}}%
    \put(0.13687763,0.13140784){\color[rgb]{0,0,0}\makebox(0,0)[lb]{\smash{$c^R$}}}%
    \put(0.05190837,0.2282744){\color[rgb]{0,0,0}\makebox(0,0)[lb]{\smash{$d_1^P$}}}%
    \put(0.01680758,0.35795164){\color[rgb]{0,0,0}\makebox(0,0)[lb]{\smash{$d_1^R$}}}%
    \put(-0.00135873,0.50754295){\color[rgb]{0,0,0}\makebox(0,0)[lb]{\smash{$d_2^P$}}}%
    \put(0.15280442,0.79329451){\color[rgb]{0,0,0}\makebox(0,0)[lb]{\smash{$d_{-1}^P$}}}%
    \put(0.3217571,0.86075833){\color[rgb]{0,0,0}\makebox(0,0)[lb]{\smash{$d_{-1}^R$}}}%
  \end{picture}%
\endgroup%

%% file: figures/E-cap-P-general.pdf_tex
\begingroup%
  \makeatletter%
  \providecommand\color[2][]{%
    \errmessage{(Inkscape) Color is used for the text in Inkscape, but the package 'color.sty' is not loaded}%
    \renewcommand\color[2][]{}%
  }%
  \providecommand\transparent[1]{%
    \errmessage{(Inkscape) Transparency is used (non-zero) for the text in Inkscape, but the package 'transparent.sty' is not loaded}%
    \renewcommand\transparent[1]{}%
  }%
  \providecommand\rotatebox[2]{#2}%
  \ifx\svgwidth\undefined%
    \setlength{\unitlength}{519.93837132bp}%
    \ifx\svgscale\undefined%
      \relax%
    \else%
      \setlength{\unitlength}{\unitlength * \real{\svgscale}}%
    \fi%
  \else%
    \setlength{\unitlength}{\svgwidth}%
  \fi%
  \global\let\svgwidth\undefined%
  \global\let\svgscale\undefined%
  \makeatother%
  \begin{picture}(1,0.82629992)%
    \put(0,0){\includegraphics[width=\unitlength]{E-cap-P-general.pdf}}%
    \put(0.26947566,0.6653329){\color[rgb]{0,0,0}\makebox(0,0)[lb]{\smash{$a^P$}}}%
    \put(0.62639578,0.73795301){\color[rgb]{0,0,0}\makebox(0,0)[lb]{\smash{$P$}}}%
    \put(0.17880799,0.23088956){\color[rgb]{0,0,0}\makebox(0,0)[lb]{\smash{$R$}}}%
    \put(0.04834551,0.43090054){\color[rgb]{0,0,0}\makebox(0,0)[lb]{\smash{$d_{\inv{q}}^P,d_{2\inv{q}}^P,\dots,d_{(p-1)\inv{q}}^P$}}}%
    \put(0.39152729,0.45445853){\color[rgb]{0,0,0}\makebox(0,0)[lb]{\smash{$b_{(q-1)\inv{p}}^P,b_{(q-2)\inv{p}}^P,\dots,b_{\inv{p}}^P$}}}%
    \put(0.3017737,0.00369635){\color[rgb]{0,0,0}\makebox(0,0)[lb]{\smash{$d_{\inv{q}}^R,d_{2\inv{q}}^R,\dots,d_{(p-1)\inv{q}}^R$}}}%
    \put(0.60460308,0.02613347){\color[rgb]{0,0,0}\makebox(0,0)[lb]{\smash{$b_{(q-1)\inv{p}}^R,b_{(q-2)\inv{p}}^R,\dots,b_{\inv{p}}^R$}}}%
    \put(0.53016404,0.21149705){\color[rgb]{0,0,0}\makebox(0,0)[lb]{\smash{$c^R$}}}%
    \put(0.07456611,0.71120076){\color[rgb]{0,0,0}\makebox(0,0)[lb]{\smash{$H_1^P$}}}%
    \put(0.26399842,0.561){\color[rgb]{0,0,0}\makebox(0,0)[lb]{\smash{$H_2^P$}}}%
    \put(0.36712132,0.28974184){\color[rgb]{0,0,0}\makebox(0,0)[lb]{\smash{$H_1^R$}}}%
    \put(0.5139594,0.13281575){\color[rgb]{0,0,0}\makebox(0,0)[lb]{\smash{$H_2^R$}}}%
  \end{picture}%
\endgroup%

%% file: figures/H-dual.pdf_tex
\begingroup%
  \makeatletter%
  \providecommand\color[2][]{%
    \errmessage{(Inkscape) Color is used for the text in Inkscape, but the package 'color.sty' is not loaded}%
    \renewcommand\color[2][]{}%
  }%
  \providecommand\transparent[1]{%
    \errmessage{(Inkscape) Transparency is used (non-zero) for the text in Inkscape, but the package 'transparent.sty' is not loaded}%
    \renewcommand\transparent[1]{}%
  }%
  \providecommand\rotatebox[2]{#2}%
  \ifx\svgwidth\undefined%
    \setlength{\unitlength}{394.94613555bp}%
    \ifx\svgscale\undefined%
      \relax%
    \else%
      \setlength{\unitlength}{\unitlength * \real{\svgscale}}%
    \fi%
  \else%
    \setlength{\unitlength}{\svgwidth}%
  \fi%
  \global\let\svgwidth\undefined%
  \global\let\svgscale\undefined%
  \makeatother%
  \begin{picture}(1,0.54247322)%
    \put(0,0){\includegraphics[width=\unitlength]{H-dual.pdf}}%
  \end{picture}%
\endgroup%

%% file: figures/D-cap-P-separating.pdf_tex

\begingroup
  \makeatletter
  \providecommand\color[2][]{%
    \errmessage{(Inkscape) Color is used for the text in Inkscape, but the package 'color.sty' is not loaded}
    \renewcommand\color[2][]{}%
  }
  \providecommand\transparent[1]{%
    \errmessage{(Inkscape) Transparency is used (non-zero) for the text in Inkscape, but the package 'transparent.sty' is not loaded}
    \renewcommand\transparent[1]{}%
  }
  \providecommand\rotatebox[2]{#2}
  \ifx\svgwidth\undefined
    \setlength{\unitlength}{401.93733896pt}
  \else
    \setlength{\unitlength}{\svgwidth}
  \fi
  \global\let\svgwidth\undefined
  \makeatother
  \begin{picture}(1,0.59806359)%
    \put(0,0){\includegraphics[width=\unitlength]{D-cap-P-separating.pdf}}%
  \end{picture}%
\endgroup

%% file: figures/SC.pdf_tex

\begingroup
  \makeatletter
  \providecommand\color[2][]{%
    \errmessage{(Inkscape) Color is used for the text in Inkscape, but the package 'color.sty' is not loaded}
    \renewcommand\color[2][]{}%
  }
  \providecommand\transparent[1]{%
    \errmessage{(Inkscape) Transparency is used (non-zero) for the text in Inkscape, but the package 'transparent.sty' is not loaded}
    \renewcommand\transparent[1]{}%
  }
  \providecommand\rotatebox[2]{#2}
  \ifx\svgwidth\undefined
    \setlength{\unitlength}{215.33273621pt}
  \else
    \setlength{\unitlength}{\svgwidth}
  \fi
  \global\let\svgwidth\undefined
  \makeatother
  \begin{picture}(1,0.33327737)%
    \put(0,0){\includegraphics[width=\unitlength]{SC.pdf}}%
    \put(-0.00223129,0.22457489){\color[rgb]{0,0,0}\makebox(0,0)[lb]{\smash{$x$}}}%
    \put(0.04917006,0.16026637){\color[rgb]{0,0,0}\makebox(0,0)[lb]{\smash{$y$}}}%
    \put(0.51432177,0.00892514){\color[rgb]{0,0,0}\makebox(0,0)[lb]{\smash{$y$}}}%
    \put(0.94352853,0.15127864){\color[rgb]{0,0,0}\makebox(0,0)[lb]{\smash{$y$}}}%
    \put(0.42527116,0.00892514){\color[rgb]{0,0,0}\makebox(0,0)[lb]{\smash{$x$}}}%
    \put(0.86567742,0.10067008){\color[rgb]{0,0,0}\makebox(0,0)[lb]{\smash{$x$}}}%
  \end{picture}%
\endgroup

%% file: figures/E-cap-P-special.pdf_tex

\begingroup
  \makeatletter
  \providecommand\color[2][]{%
    \errmessage{(Inkscape) Color is used for the text in Inkscape, but the package 'color.sty' is not loaded}
    \renewcommand\color[2][]{}%
  }
  \providecommand\transparent[1]{%
    \errmessage{(Inkscape) Transparency is used (non-zero) for the text in Inkscape, but the package 'transparent.sty' is not loaded}
    \renewcommand\transparent[1]{}%
  }
  \providecommand\rotatebox[2]{#2}
  \ifx\svgwidth\undefined
    \setlength{\unitlength}{507.13490453pt}
  \else
    \setlength{\unitlength}{\svgwidth}
  \fi
  \global\let\svgwidth\undefined
  \makeatother
  \begin{picture}(1,0.84774577)%
    \put(0,0){\includegraphics[width=\unitlength]{E-cap-P-special.pdf}}%
    \put(0.27624963,0.68271343){\color[rgb]{0,0,0}\makebox(0,0)[lb]{\smash{$a^P$}}}%
    \put(0.64218082,0.75716695){\color[rgb]{0,0,0}\makebox(0,0)[lb]{\smash{$P$}}}%
    \put(0.1832929,0.23730183){\color[rgb]{0,0,0}\makebox(0,0)[lb]{\smash{$R$}}}%
    \put(0.04376555,0.44482114){\color[rgb]{0,0,0}\makebox(0,0)[lb]{\smash{$d_1^P,d_{2}^P,\dots,d_{-1}^P$}}}%
    \put(0.33590719,0.45211058){\color[rgb]{0,0,0}\makebox(0,0)[lb]{\smash{$b_2^P,b_{4}^P,\dots,b_{-1}^P$}}}%
    \put(0.29233946,0.0096108){\color[rgb]{0,0,0}\makebox(0,0)[lb]{\smash{$d_1^R,d_{2}^R,\dots,d_{-1}^R$}}}%
    \put(0.5923382,0.00249564){\color[rgb]{0,0,0}\makebox(0,0)[lb]{\smash{$b_2^R,b_{4}^R,\dots,b_{-1}^R$}}}%
    \put(0.54351955,0.21741973){\color[rgb]{0,0,0}\makebox(0,0)[lb]{\smash{$c^R$}}}%
    \put(0.51084817,0.49319862){\color[rgb]{0,0,0}\makebox(0,0)[lb]{\smash{$b_1^P,b_{3}^P,\dots,b_{-2}^P$}}}%
    \put(0.7813313,0.0532033){\color[rgb]{0,0,0}\makebox(0,0)[lb]{\smash{$b_1^R,b_{3}^R,\dots,b_{-2}^R$}}}%
    \put(0.07269568,0.72972669){\color[rgb]{0,0,0}\makebox(0,0)[lb]{\smash{$H_1^P$}}}%
    \put(0.24817,0.57651403){\color[rgb]{0,0,0}\makebox(0,0)[lb]{\smash{$H_2^P$}}}%
    \put(0.3647164,0.30020741){\color[rgb]{0,0,0}\makebox(0,0)[lb]{\smash{$H_1^R$}}}%
    \put(0.51269101,0.1365187){\color[rgb]{0,0,0}\makebox(0,0)[lb]{\smash{$H_2^R$}}}%
  \end{picture}%
\endgroup

%% file: figures/alpha-E.pdf_tex
\begingroup%
  \makeatletter%
  \providecommand\color[2][]{%
    \errmessage{(Inkscape) Color is used for the text in Inkscape, but the package 'color.sty' is not loaded}%
    \renewcommand\color[2][]{}%
  }%
  \providecommand\transparent[1]{%
    \errmessage{(Inkscape) Transparency is used (non-zero) for the text in Inkscape, but the package 'transparent.sty' is not loaded}%
    \renewcommand\transparent[1]{}%
  }%
  \providecommand\rotatebox[2]{#2}%
  \ifx\svgwidth\undefined%
    \setlength{\unitlength}{353.61549072bp}%
    \ifx\svgscale\undefined%
      \relax%
    \else%
      \setlength{\unitlength}{\unitlength * \real{\svgscale}}%
    \fi%
  \else%
    \setlength{\unitlength}{\svgwidth}%
  \fi%
  \global\let\svgwidth\undefined%
  \global\let\svgscale\undefined%
  \makeatother%
  \begin{picture}(1,0.89819064)%
    \put(0,0){\includegraphics[width=\unitlength]{alpha-E.pdf}}%
    \put(0.50422562,0.87756442){\color[rgb]{0,0,0}\makebox(0,0)[lb]{\smash{$a^P$}}}%
    \put(0.63893328,0.86557437){\color[rgb]{0,0,0}\makebox(0,0)[lb]{\smash{$b_1^R$}}}%
    \put(0.75793944,0.80482257){\color[rgb]{0,0,0}\makebox(0,0)[lb]{\smash{$b_1^P$}}}%
    \put(0.84514395,0.70803521){\color[rgb]{0,0,0}\makebox(0,0)[lb]{\smash{$b_2^R$}}}%
    \put(0.89862574,0.58771595){\color[rgb]{0,0,0}\makebox(0,0)[lb]{\smash{$b_2^P$}}}%
    \put(0.91621151,0.46840249){\color[rgb]{0,0,0}\makebox(0,0)[lb]{\smash{$b_3^R$}}}%
    \put(0.83234425,0.195732){\color[rgb]{0,0,0}\makebox(0,0)[lb]{\smash{$b_{-3}^P$}}}%
    \put(0.72872864,0.09508778){\color[rgb]{0,0,0}\makebox(0,0)[lb]{\smash{$b_{-2}^R$}}}%
    \put(0.54917514,0.01350489){\color[rgb]{0,0,0}\makebox(0,0)[lb]{\smash{$b_{-2}^P$}}}%
    \put(0.38091349,0.0035791){\color[rgb]{0,0,0}\makebox(0,0)[lb]{\smash{$b_{-1}^R$}}}%
    \put(0.23490239,0.05202398){\color[rgb]{0,0,0}\makebox(0,0)[lb]{\smash{$b_{-1}^P$}}}%
    \put(0.13687763,0.13140784){\color[rgb]{0,0,0}\makebox(0,0)[lb]{\smash{$c^R$}}}%
    \put(0.05190837,0.2282744){\color[rgb]{0,0,0}\makebox(0,0)[lb]{\smash{$d_1^P$}}}%
    \put(0.01680758,0.35795164){\color[rgb]{0,0,0}\makebox(0,0)[lb]{\smash{$d_1^R$}}}%
    \put(-0.00135873,0.50754295){\color[rgb]{0,0,0}\makebox(0,0)[lb]{\smash{$d_2^P$}}}%
    \put(0.15280442,0.79329451){\color[rgb]{0,0,0}\makebox(0,0)[lb]{\smash{$d_{-1}^P$}}}%
    \put(0.3217571,0.86075833){\color[rgb]{0,0,0}\makebox(0,0)[lb]{\smash{$d_{-1}^R$}}}%
  \end{picture}%
\endgroup%